\documentclass[letterpaper, 11pt,  reqno]{amsart}

\usepackage{amsmath,amssymb,amscd,amsthm,amsxtra, esint, mathtools, xcolor}
\usepackage{latexsym}
\usepackage{skak}
\usepackage{accents, cancel}
\usepackage{bbm}
\usepackage{enumerate}

\usepackage[implicit=true]{hyperref}

\usepackage{stmaryrd}
\usepackage{xsavebox}

\usepackage{tikz}
\usetikzlibrary{calc,intersections,through,backgrounds}
\usetikzlibrary{decorations.pathreplacing,decorations.pathmorphing,decorations.markings,decorations.shapes}
\usetikzlibrary{shapes}

\newsavebox{\sticka}
\savebox{\sticka}[4pt]{
\begin{tikzpicture}[x=1pt, y=1pt]
\draw[line width= 1] (0,0) -- (0,6); 
\path [draw, fill] (0,7.5) circle (1.5pt); 
\end{tikzpicture}
}

\allowdisplaybreaks[2]

\newtheorem{theorem}{Theorem} [section]

\newtheorem{lemma}[theorem]{Lemma}
\newtheorem{proposition}[theorem]{Proposition}

\theoremstyle{definition}
\newtheorem{remark}[theorem]{Remark}

\newcommand{\noi}{\noindent}
\newcommand{\Z}{\mathbb{Z}}
\newcommand{\R}{\mathbb{R}}
\newcommand{\C}{\mathbb{C}}
\newcommand{\T}{\mathbb{T}}

\newcommand{\X}{ X^\alpha}
\renewcommand{\H}{\mathcal H}
\newcommand{\W}{\mathcal W}

\renewcommand{\u}{{\mathbf u}}
\renewcommand{\v}{{\mathbf v}}
\newcommand{\w}{{\mathbf w}}

\newcommand{\vect}[2]{\begin{pmatrix} #1 \\ #2 \end{pmatrix}}

\def\stick{\mathchoice
    {\raisebox{-1pt}{\usebox{\sticka}}}%
    {\raisebox{-1pt}{\usebox{\sticka}}}%
    {}%
    {}}

\let\Re=\undefined\DeclareMathOperator*{\Re}{Re}
\let\Im=\undefined\DeclareMathOperator*{\Im}{Im}

\newcommand{\prob}{\mathbb{P}}  
\newcommand{\Pt}[1]
{\mathcal{P}_{#1}}

\newcommand{\Q}{\mathcal{Q}}

\newcommand{\E}{\mathbb{E}}

\renewcommand{\L}{\mathcal{L}}

\newcommand{\B}{\mathcal{B}}

\newcommand{\F}{\mathcal{F}}

\newcommand{\al}{\alpha}
\newcommand{\be}{\beta}
\newcommand{\dl}{\delta}

\newcommand{\nb}{\nabla}

\newcommand{\vp}{\varphi}
\newcommand{\Dl}{\Delta}
\newcommand{\ep}{\varepsilon}

\newcommand{\g}{\gamma}

\newcommand{\Ld}{\Lambda}
\newcommand{\s}{\sigma}

\newcommand{\ft}{\widehat}

\newcommand{\wt}{\widetilde}
\newcommand{\cj}{\overline}

\newcommand{\dt}{\partial_t}

\newcommand{\embeds}{\hookrightarrow}

\DeclareMathOperator{\Law}{Law}

\newcommand{\ta}{\theta}

\renewcommand{\l}{\ell}
\renewcommand{\o}{\omega}
\renewcommand{\O}{\Omega}

\newcommand{\les}{\lesssim}
\newcommand{\ges}{\gtrsim}

\newcommand{\jb}[1]
{\langle #1 \rangle}

\newcommand{\sbb}[1]
{[ #1 ]}

\newcommand{\snb}{\sbb{\nb}}

\newcommand{\ind}{\mathbf 1}

\renewcommand{\S}{\mathcal{S}}

\newcommand{\pa}{\partial}

\newcommand{\N}{\mathbb{N}}

\newcommand{\CC}{\mathcal{C}}

\newcommand{\EE}{\mathcal{E}}

\renewcommand{\H}{\mathcal{H}}

\newcommand{\Wal}{\mathcal{W}^{\al,\frac{2}{\al}}}

\newtheorem*{ackno}{Acknowledgements}

\newcommand{\eps}{\ep}

\newcommand{\Dr}{\Theta}
\newcommand{\U}{\mathcal{U}}

\usepackage{tikz} 
\usepackage{marginnote}
\usepackage{scalerel} 

\usetikzlibrary{shapes.misc}
\usetikzlibrary{shapes.symbols}
\usetikzlibrary{decorations}
\usetikzlibrary{decorations.markings}

\makeatletter

\def\<#1>{\xusebox{#1}}
\makeatother

\newcommand{\pe}{\mathbin{\scaleobj{0.7}{\tikz \draw (0,0) node[shape=circle,draw,inner sep=0pt,minimum size=8.5pt] {\scriptsize  $=$};}}}

\newcommand{\pl}{\mathbin{\scaleobj{0.7}{\tikz \draw (0,0) node[shape=circle,draw,inner sep=0pt,minimum size=8.5pt] {\scriptsize $<$};}}}
\newcommand{\pg}{\mathbin{\scaleobj{0.7}{\tikz \draw (0,0) node[shape=circle,draw,inner sep=0pt,minimum size=8.5pt] {\scriptsize $>$};}}}

\tikzset{
    position/.style args={#1:#2 from #3}{
        at=(#3.#1), anchor=#1+180, shift=(#1:#2)
    }
}

\numberwithin{equation}{section}
\numberwithin{theorem}{section}

\DeclareFontFamily{U}{matha}{\hyphenchar\font45}
\DeclareFontShape{U}{matha}{m}{n}{
      <5> <6> <7> <8> <9> <10> gen * matha
      <10.95> matha10 <12> <14.4> <17.28> <20.74> <24.88> matha12
      }{}
\DeclareSymbolFont{matha}{U}{matha}{m}{n}
\DeclareFontSubstitution{U}{matha}{m}{n}

\DeclareFontFamily{U}{mathx}{\hyphenchar\font45}
\DeclareFontShape{U}{mathx}{m}{n}{
      <5> <6> <7> <8> <9> <10>
      <10.95> <12> <14.4> <17.28> <20.74> <24.88>
      mathx10
      }{}
\DeclareSymbolFont{mathx}{U}{mathx}{m}{n}
\DeclareFontSubstitution{U}{mathx}{m}{n}

\DeclareMathDelimiter{\vvvert}{0}{matha}{"7E}{mathx}{"17}

\begin{document}
\baselineskip = 14pt

\title[Quasi-invariance for SDNLW on $\T^{2}$]
{Quasi-invariance of the Gaussian measure for the two-dimensional stochastic cubic nonlinear wave equation}

\author[J.~Forlano and L.~Tolomeo]
{Justin Forlano and Leonardo Tolomeo}
\dedicatory{Dedicated to the memory of Professor Giuseppe Da Prato}

\address{
Justin Forlano\\ School of Mathematics\\
The University of Edinburgh\\
and The Maxwell Institute for the Mathematical Sciences\\
James Clerk Maxwell Building\\
The King's Buildings\\
Peter Guthrie Tait Road\\
Edinburgh\\ 
EH9 3FD\\
United Kingdom}

\email{j.forlano@ed.ac.uk}

\address{
Leonardo Tolomeo\\ School of Mathematics\\
The University of Edinburgh\\
and The Maxwell Institute for the Mathematical Sciences\\
James Clerk Maxwell Building\\
The King's Buildings\\
Peter Guthrie Tait Road\\
Edinburgh\\ 
EH9 3FD\\
United Kingdom}

\email{l.tolomeo@ed.ac.uk}




\begin{abstract}
We consider the stochastic damped nonlinear wave equation $\partial_t^{2}u+\partial_t u+u-\Delta u +u^{3} = \sqrt{2} {\langle{\nabla}\rangle^{-s}} \xi$ on the two-dimensional torus $\mathbb T^2$, where $\xi$ denotes a space-time white noise and $s>0$. 
We show that the measure $\vec{\mu}_s$ corresponding to the unique invariant measure for the flow of the associated linear equation is quasi-invariant under the nonlinear stochastic flow.  
\end{abstract}

\maketitle
\tableofcontents

\section{Introduction}

We consider here the following stochastic damped nonlinear wave equation (SDNLW): 
\begin{equation}
\begin{cases}\label{SDNLW}
\dt^{2}u+\dt u+u-\Delta u +u^{3}=\sqrt{2} \jb{\nb}^{-s} \xi \\
(u,\dt u)|_{t = 0} = (u_0,u_1)
\end{cases}
\qquad ( t, x) \in \R_{+} \times \T^{2}, 
\end{equation}
where $s>0$, $\xi$ is a space-time white noise, and $\jb{\nb}^{-s}$ denotes the smoothing operator 
\begin{equation}
\jb{\nb}^{-s}:=(1-\Dl)^{-\frac{s}{2}}. \label{Bs}
\end{equation}
In the following, we consider \eqref{SDNLW} as a first-order system in the variable $\u := (u, v)^{\top}$,
\begin{equation}
\begin{cases}
\partial_t \u =  \partial_t\begin{pmatrix} u \\ v \end{pmatrix}  = -\begin{pmatrix} 0 & -1 \\  1-\Delta & 1 \end{pmatrix} \begin{pmatrix} u \\ v \end{pmatrix}  - \vect{0}{u^3} +\vect{0}{\sqrt{2} \jb{\nb}^{-s}\xi}, \\ 
\u|_{t = 0} = \u_0 = \begin{pmatrix} u_0 \\ u_1 \end{pmatrix}.
\end{cases}
 \label{NLWvec}
\end{equation}

The interest for the equation stemmed from recent results pertaining to the PDE construction of the Euclidean $\Phi^4$ quantum field theory \cite{dpd, BG, BG2, gh1, gh2, mw2}.
Namely, if one considers the stochastic quantisation equation
\begin{equation} \label{SQE}
\partial_t u + u - \Delta u = u^3 + \sqrt2\xi, 
\end{equation}
as time goes to infinity, we expect a generic solution to \eqref{SQE} to converge to the (massive) $\Phi^4$-measure.

In the above context, the stochastic wave equation 
\begin{equation}\label{CSQE}
\partial_t^2 u + \partial_t u + u - \Delta u = u^3 + \sqrt2 \xi, 
\end{equation}
corresponds to the so-called canonical stochastic quantisation equation, or equivalently, to the kinetic Langevin equation for the $\Phi^4$-measure with momentum $v = \partial_t u$. Therefore, one can sample the $\Phi^4$-measure by sampling the law of the first component of the solution of \eqref{CSQE}. This procedure has the name of Hamiltonian Montecarlo (HCM). 
It has numerically been observed that in the finite dimensional setting, HCM converges faster than MCMC (Markov Chain Montecarlo) in many situations, see \cite{EGZ} (and references within), for a rigorous justification of these faster convergence rates.

As a consequence, equation \eqref{CSQE} has received a substantial amount of attention in recent years, as an alternative approach to stochastic quantisation of the $\Phi^4$ measure. However, due to the worse analytical properties of the wave propagator $(\partial_t^2 - \Delta)^{-1}$ compared to those of the heat propagator $(\partial_t - \Delta)^{-1}$, there are far fewer results available, namely, uniqueness and convergence to the invariant measure have been proven by the second author only in \cite{t18erg} for dimension $d = 1$ and in \cite{Tolomeo3} for dimension $d = 2$. Local and global well posedness results for related models, but without uniqueness of the measure, can be found in \cite{GKO, GKO2, t18wave, OOT1, OOT2, OWZ}. 

In light of the above discussion, we see the equation \eqref{SDNLW} as a version of \eqref{CSQE} in fractional dimension $d = 2-2 s$. A similar approach to fractional dimension has already been considered for the parabolic case of \eqref{SQE}, see for instance \cite{cmw}. 
In the work by the authors \cite{FT1}, it has been shown that equation \eqref{SDNLW} also has a unique invariant measure, to which we have convergence for every initial data belonging to a certain class. However, strictly speaking, the equation \eqref{SDNLW} is not a Langevin equation, because of the presence of the smoothing operator $\jb{\nabla}^{-s}$. 
This prevents us from writing an explicit formula for the invariant measure, which in turn leaves many open questions. A fairly natural question to ask is the following. Denoting by $\vec{\mu}_s$ the invariant measure for the linear equation 
$$ \dt^{2}u+\dt u+u-\Delta u =\sqrt{2} \jb{\nb}^{-s} \xi, $$
or more precisely, to the system 
\begin{equation*}
\begin{cases}
\partial_t \u =  \partial_t \begin{pmatrix} u \\ v \end{pmatrix}  = -\begin{pmatrix} 0 & -1 \\  1-\Delta & 1 \end{pmatrix} \begin{pmatrix} u \\ v \end{pmatrix}  +\vect{0}{\sqrt{2} \jb{\nb}^{-s}\xi}, \\ 
\u|_{t = 0} = \u_0 = \begin{pmatrix} u_0 \\ u_1 \end{pmatrix}.
\end{cases}
\end{equation*}
is it true that the invariant measure $\vec{\rho}_s$ of \eqref{SDNLW} satisfies $\vec{\rho}_s \ll \vec{\mu}_s$? As opposed to the invariant measure for the nonlinear equation, the measure $\vec{\mu}_s$ has an explicit formula, namely, it formally holds that 
\begin{align}
 \vec{\mu}_s = Z^{-1} \exp\Big(- \frac12 \|u\|_{H^{1+s}}^2 - \frac12 \|v\|_{H^s}^2\Big). \label{mus}
\end{align}
Therefore, this absolute continuity property would allow us to bootstrap relevant properties of the invariant measure from the analogous properties of $\vec{\mu}_s$. However, while the analogous of this statement is relatively simple to show for the parabolic equation \eqref{SQE}, for the wave case, it turns out to be a very difficult problem. Indeed, even in the case of the stochastic Navier-Stokes equation on $\T^2$,
\begin{equation}
\partial_t u = \Delta u - u \cdot \nabla u + \jb{\nabla}^{-s} \xi \label{SNSE},
\end{equation}
this question is still open since the original proof of the existence of the invariant measure in \cite{FM95, Fer97}. 
Indeed, the current state-of-the art is that one can show that the invariant measure is absolutely continuous with respect to the stationary law of the linear solution when $\Delta$ is replaced with $-(-\Delta)^{1+\eps}$ for some $\eps > 0$, see \cite{MS}.

A closely related question is the following: if $\Law(\u_0) = \vec{\mu}_s,$ is it true that $\Law(\u(t)) \ll \vec{\mu}_s$? If this happens, we say that the measure $\vec{\mu}_s$ is \emph{quasi-invariant} under the (stochastic) flow of \eqref{NLWvec}. In the parabolic case of \eqref{SQE} and \eqref{SNSE}, or more in general, whenever one has the strong Feller property, the two questions are actually equivalent. Unfortunately, the strong Feller property fails for equation \eqref{NLWvec}, and one cannot deduce absolute continuity of the invariant measure from quasi-invariance.

Quasi-invariance for the deterministic flow of Hamiltonian and dispersive PDEs in situations where the initial data is {not} distributed according to an invariant measure has been studied extensively in the last decade. This pursuit was initiated by Tzvetkov in \cite{Tz} for the BBM equation, with many results following for a number of other Hamiltonian PDEs, see \cite{Tz, OTz, OTz2, OST, OTT, PTV, FT, GOTW, STX, DT2020,  GLT1, GLT2, PTV2, BTh, FS, OS, FT2, ST, CT, Knezevitch} and references therein for a list of results in this direction. We point out the results in \cite{OTz2, GOTW, STX} prove quasi-invariance results in the case of the deterministic and undamped version of \eqref{SDNLW}.

Nevertheless, the question of quasi-invariance for stochastic dispersive models is yet to be answered.  
In this paper, we prove the first of such results; in particular, we show the following.  

\begin{theorem}\label{THM:QI}
For each $s>0$, let $\Phi_{t}(\u_0;\xi)$ denote the solution to \textup{SDNLW}~\ref{SDNLW} with $\Law(\u_0 )= \vec{\mu}_{s}$. Then for each $t\geq 0$, the probability measure $\textup{Law}(\Phi_{t}(\u_0;\xi))$ is absolutely continuous with respect to $\vec{\mu}_{s}$.
\end{theorem}

\subsection{Outline of the proof}
The proof of Theorem \ref{THM:QI} will follow closely the strategy of the deterministic results in \cite{OTz2, GOTW, STX}, hence we first describe how one obtains quasi-invariance of the measure $\vec{\mu}_s$ for the deterministic equation 
\begin{equation} \label{det_wave}
\dt^2 u + (1 - \Delta) u + u^3 = 0.
\end{equation}
One can formally see this equation as {an} infinite-dimensional system of ODEs in the form 
$$ \dot y = b(y), $$
with $\text{div}(b) = 0$ on some infinite-dimensional vector space $X$. Therefore, denoting by $\Phi_t$ the flow of this system of ODEs, for a measure $\nu_0 = \exp\big( - E(x)\big)dx$ {and} for any Borel set $A$, one has that 
$$ \frac{d}{dt} \nu_0(\Phi_t(A)) = \int_{\Phi_t(A)} - \frac{d}{dt}\left. E(\Phi_t(x))\right|_{t=0} d\nu_0(x). $$
This allows to perform a Gronwall-type argument for the quantity $\nu_0(\Phi_t(E))$, from which one obtains that quasi-invariance holds whenever for every ball $B_R \subset X$, there exists some $\eta_R > 0$ such that
\begin{equation}
 \int_{B_X(0, R)} \exp\Big(\eta_R \big|\frac{d}{dt}\left. E(\Phi_t(x))\right|_{t=0}\big|\Big) d \nu_0 < \infty.  \label{exp_integrability}
\end{equation}
While not codified in this way, one can find a proof of this result in \cite{PTV}. {As} the condition \eqref{exp_integrability} does not hold for the measures $\vec{\mu}_s$, what is shown in \cite{OTz2, GOTW, STX} is that, for an appropriate range of $s>0$, if one defines the energy functional 
$$ E(u,v) = \frac12 \|u\|_{H^{1+s}}^2 - \frac12 \|v\|_{H^s}^2 - \frac32 \int \big(\jb{\nabla}^s u - \s\big) u^2   $$ 
with $\s = \E_{\vec{\mu}_s} \|\jb{\nabla}^s u\|_{L^2}^2$, then 
$$ \vec{\mu}_s \ll Z^{-1} \exp(- E(u,v)) dudv \ll \vec{\mu}_s ,  $$
and the property \eqref{exp_integrability} holds.\footnote{The expression above must be considered formally, as $\s=+\infty$. Nevertheless, one can define the measure $Z^{-1} \exp(- E(u,v)) dudv$ via a renormalisation procedure.}

In order to prove Theorem \ref{THM:QI}, we follow a similar approach in the stochastic setting. We define the measure 
$$ \vec{\nu}_s \sim \exp\Big( - \frac12 \|u\|_{H^{1+s}}^2 - \frac12 \|v\|_{H^s}^2 - \frac32 \int \big(\jb{\nabla}^s u - \s\big) u^2 \Big)dudv, $$
and we try to show that under a condition analogous to \eqref{exp_integrability}, we can deduce quasi-invariance of the measure $ \vec{\nu}_s $ in the sense of Theorem \ref{THM:QI}. Indeed, in Section 5, we first prove that the measure $\vec{\nu}_s$  satisfies the alternative condition
\begin{equation} \label{exp_integrability_2}
 \int_{B_X(0, R)} \exp\Big(\eta_R \big|\frac{d}{dt}\left. \E_\xi\big[E(\Phi_t(\u_0, \xi))\big]\right|_{t=0}\big|\Big) d \nu_0(\u_0) < \infty, 
\end{equation}
and then we prove that this condition can indeed replace \eqref{exp_integrability} in the stochastic setting, thus completing the proof.
While the argument is specialised to equation \eqref{SDNLW}, we believe that a similar condition can be proven more in general. 




\subsection{Further remarks}
\begin{remark}\rm
In the proof of the condition \eqref{exp_integrability_2}, we follow a slightly different approach from the cited works \cite{OTz2, STX}. This allows us to prove that the estimate holds in the full non-singular regime $s >0$. When applied to the deterministic wave equation \eqref{det_wave}, this allows us to show the estimate \eqref{exp_integrability}, and thus quasi-invariance of $\vec{\mu}_s$, in the regime $s>0$, thus improving the result in \cite{STX}, which instead requires $s > \frac 52$. 
\end{remark}

\begin{remark}\rm 
In the deterministic case \eqref{det_wave}, it is in principle possible to write an explicit formula for the Radon-Nikodym derivative $\frac{d (\Phi_t)_\ast \vec{\mu}_s}{d \vec{\mu}_s}$. While this formula has not been used in the context of wave equations, for other deterministic models, it allowed to improve quasi-invariance results. See \cite{DT2020, FS, FT2, GLT1, GLT2, CT, Knezevitch} for some examples. However, in the stochastic case, we do not have such an explicit formula for the density, as we cannot solve the Fokker-Plank equation \eqref{FP} explicitly.
\end{remark}

\begin{remark}\rm
A very recent result by Hairer-Kusuoka-Nagoji \cite{HKN} shows the reverse of Theorem \ref{THM:QI} for a wide class of parabolic equations, i.e.\ that under certain assumptions, the invariant measure for {singular} parabolic SPDEs is singular with respect to the law of the stationary linear solution. 
Interestingly, the main assumption of their result, i.e.\ the fact that $\Phi_t - \Phi_t^{\text{lin}}$ does not belong to the Cameron-Martin space for $\vec{\mu_s}$, also applies to equation \eqref{SDNLW}. Nevertheless, the result of this paper shows that the measure is quasi-invariant. The discrepancy between these results is due to the following.
\begin{itemize}
\item[-] The result in \cite{HKN} heavily relies on the fact that the nonlinearity of the equation is not an $L^1_{\text{loc}}$ function when computed in the linear solution, and 
\smallskip

\item[-] The dispersive/Hamiltonian structure of the wave equation introduces further cancellations that are not present in the parabolic case. 
\end{itemize}

\end{remark}
\section{Preliminaries}

\noi
{\bf Notations.}
For a function $f: \T^2 \sim [0,1]^2 \to \C$, its Fourier series is given by $\widehat f: \Z^2 \to \C$, where
\begin{equation*}
\widehat f (n) := \int_{\T^2} f(x) e^{-2\pi i n \cdot x} d x \quad \text{and} \quad f(x) := \sum_{n\in\Z^2} \widehat f(n) e^{2\pi i n \cdot x}.
\end{equation*}
Moreover, for a function $F: \C^2 \to \C$, the expression $F(\nabla)u$ will denote the distribution with Fourier coefficients given by
\begin{equation*}
\widehat{F(\nabla) u}(n) := F(2 \pi i n) \widehat u (n).
\end{equation*}
For $x \in \R^2$, we denote $\jb{x} := \big(1 + |x|^2\big)^\frac 12 $.
Given $N\in \N_{0}:=\N\cup \{0\}$, we define $\Pi_{\leq N}$ as the sharp projection onto frequencies in the square $[-N,N]^{2}$, that is, 
\begin{align*}
\Pi_{\leq N}u(x)= \frac{1}{(2\pi)^{2}} \sum_{\substack{n\in \Z^{2}\\ |n|_{\infty}\leq N}} \ft{u}(n)e^{in\cdot x},
\end{align*}
where for $n=(n_1,n_2)$, $|n|_{\infty}:= \max_{j}|n_j|$,
and similarly, when $\u$ is vector valued,
\begin{align*}
\Pi_{\leq N}\u(x)= ( \Pi_{\leq N} \pi_1 \u (x), \Pi_{\leq N} \pi_2 \u(x)).
\end{align*}
We also define $\Pi_{>N}:=\text{Id}-\Pi_{\leq N}$ and for $N=-1$, extend these projections to $\Pi_{\leq -1}=0$ and $\Pi_{>-1}=\text{Id}$. As $\Pi_{\le N}$ is the Fourier restriction to a cube with sides of length $2N+1$, we can rewrite $\Pi_{\le N}$ as the composition of the Hilbert transform on each of the variables. Thus, from the boundedness of the Hilbert transform on $L^p(\T)$ for $1 < p < \infty$, we obtain
$$ \| \Pi_{\le N} u\|_{L^p(\T^2)} \les \|u \|_{L^p(\T^2)}, $$
 uniformly in the parameter $N$ and for $1< p<\infty$. When $N=+\infty$, we set $\Pi_{\leq +\infty}:=\text{Id}$.
 
 Given a vector $\u=(u,v)^{\top}$, we use $\pi_1$ and $\pi_2$ to denote the projection onto the first and second components of $\u$, respectively. Namely, $\pi_1 \u=u$ and $\pi_2 \u =v$.


\subsection{Function spaces}
We define the  Sobolev spaces $W^{\al,p}$ and $\W^{\al,p}:=W^{\al,p}\times W^{\al-1,p}$ via the norms
\begin{align*}
\| u\|_{W^{\al,p}}:= \| \jb{\nabla}^\al u\|_{L^p} \quad \text{and} \quad \| \u \|_{\W^{\al, p}}^p := \| \jb{\nabla}^\al u\|_{L^p}^p +  \| \jb{\nabla}^{\al-1}v\|_{L^p}^p,
\end{align*}
with the usual modification when $p=\infty$. Moreover, when $p = 2$, we write $\H^\al := \W^{\al,2}$.
With this definition and \eqref{linsol}, it holds that for every $t \in \R$,
\begin{align}
\| S(t)\u \|_{\H^{\al}} \les e^{-\frac{t}{2}}\| \u \|_{\H^{\al}}. \label{SonH}
\end{align}

Following \cite{FT1}, for $0<\al < 1$, we define the space 
\begin{align*}
\overline{X}^{\al}:=\{ \u \in \H^\al \, : \, S(t)\u \in C([0,+\infty); \W^{\al, \frac{2}{\al}}), \, \|S(t)\u\|_{\Wal}\les e^{-\frac{t}{8}} \},
\end{align*}
with the norm
  \begin{align}
  \| \u \|_{\X}:=\sup_{t\ge0} \, e^{\frac{t}{8}}\|S(t)\u\|_{\W^{\al,\frac{2}{\al}}}.
  \label{Xalpha}
\end{align}
 We define $X^{\al}$ to be the closure of trigonometric polynomials in $\overline{X}^{\al}$ under the norm $\|\cdot \|_{\X}$, which makes $(\X, \| \cdot\|_{\X})$ into a Polish space.
Note that for any $t\geq 0$, 
\begin{align}
\| S(t)\|_{X^{\al}\mapsto \X}\leq e^{-\frac{t}{8}}. \label{S(t)onX}
\end{align} 
Moreover the following embeddings hold (see \cite{FT1}):
\begin{align*}
\H^{1} \embeds X^{\al} \embeds \mathcal{W}^{0,6}.
\end{align*}
When $\al \geq 1$, we simply define $X^{\al}:=\H^{\al}$, and, when $\al>1$, the embedding $\H^{\al}\embeds L^{\infty}$ implies that $X^{\al}$ is an algebra. 

Similar to the spaces $\cj{X}^{\al}$ and $\X$, given $\al>0$, we let
\begin{align*}
\overline{Z}^{\al}:=\{ \u \in \H^{\al} \, :\, S(t)\u \in C([0,+\infty); \CC^{\al}), \,\, \| S(t)\u\|_{\CC^{\al}}\les e^{-\frac{t}{8}}\, \},
\end{align*}
and define $Z^{\al}$ to be the closure of trigonometric polynomials in $\cj{Z}^{\al}$ under the norm 
\begin{align}
\|\u\|_{Z^{\al}}=\sup_{t\geq 0} e^{\frac{t}{8}}\|S(t)\u\|_{\CC^{\al}}. \label{Znorm}
\end{align}
Now, $(Z^{\al},\|\cdot\|_{Z^{\al}})$ is a Polish space (see \cite[Lemma 1.2]{t18erg}) and $Z^{\al}\embeds X^{\al}$.


%
%

We define Littlewood-Payley projections as follows.
Let $\phi:\R\to [0,1]$ be a smooth bump function supported on $[-\frac 85, \frac 85]$ and $\phi =1$ on $[-\frac 54, \frac 54]$. For $\xi\in \R^2$, we set $\phi_{0}(\xi):=\phi(|\xi|)$ and $\phi_{j}(\xi) = \phi_{0}(2^{-j}\xi)-\phi_0 (2^{-j+1}\xi)$ for $j\in \N$. For $j\in \N_{0}$, we define the Littlewood-Payley projector $P_{N}$ as the Fourier multiplier operator with symbol
\begin{align*}
\vp_{N}(\xi) = \frac{\phi_{N}(\xi)}{\sum_{\l\in \N_0} \phi_{\l}(\xi)},
\end{align*}
where $N:=2^{j}$. 
This gives a decomposition of a distribution 
\[ f = \sum_{N\in 2^{\N_0}} P_N f.\]
We recall the definition of the Besov spaces $B^{\al}_{p,q}(\T^2)$ which are defined by norm:
\begin{align}
\| u\|_{B^{\al}_{p,q}} = \big\| N^{\al} \|P_{N} u\|_{L^{p}_{x}}  \big\|_{\l^{q}_{N}}.
\end{align}
We have $B^{\al}_{2,2}=H^{\al}$ for any $\al\in \R$, and we define the Besov-H\"{o}lder spaces $C^{\al}:=B^{\al}_{\infty,\infty}$. We also define vector versions of these spaces as $\CC^{\al}:=C^{\al}\times C^{\al-1}$ with the norm 
\begin{align*}
\| \u \|_{\CC^{\al}} =  \|\pi_1 \u\|_{C^{\al}}+\| \pi_2 \u\|_{C^{\al-1}}.
\end{align*}
The Besov spaces have the following properties. See for instance \cite{BCD}.

\begin{lemma}
The following estimates hold.\\
\textup{(i)} \textup{(interpolation)} For $0<s_1<s_2$ and $\ta =\frac{s_1}{s_2}$, it holds that
\begin{align}
\| u\|_{H^{s_1}} \les \|u\|_{H^{s_2}}^{\ta}\|u\|_{L^2}^{1-\ta}.
\end{align}
\textup{(ii)} \textup{(embeddings)} Let $s_1,s_2\in \R$ and $p_1,p_2, q_1,q_2\in [1,\infty]$. Then,
\begin{align}
\begin{split}
 \|u\|_{B^{s_1}_{p_1,q_1}} & \les \| u\|_{B^{s_2}_{p_2,q_2}} \quad \text{for} \,\, s_1 \leq s_2, p_1\leq p_2,\,\, \text{and} \,\, q_1\geq q_2, \\
 \|u\|_{B^{s_1}_{p_1,q_1}} &\les \|u\|_{B^{s_2}_{p_1,\infty}} \quad \text{for} \,\, s_1<s_2,\\
 \|u\|_{B^{0}_{p_1,\infty}} &\les \|u\|_{L^{p_1}} \les \|u\|_{B^{0}_{p_1,1}}.
\end{split} \label{Besovembeds}
\end{align}
\textup{(iii)} \textup{(algebra property)} For any $\s>0$, 
\begin{align}
\| uv \|_{C^{\s}} \les \|u\|_{C^\s}\|v\|_{C^{\s}}. \label{algebra}
\end{align}
\textup{(iv)} \textup{(duality)} Let $\s \in \R$ and $p,p',q,q'\in [1,\infty]$ such that $\frac{1}{p}+\frac{1}{p'}=\frac{1}{q}+\frac{1}{q'}=1$. Then,
\begin{align}
\bigg|\int_{\T^2}uv dx \bigg| \leq \|u\|_{B^{\s}_{p,q}}\|v\|_{B^{-\s}_{p',q'}}. \label{duality}
\end{align}
\textup{(iv)} \textup{(fractional Leibniz rule)} Let $p,p_1,p_2,p_3,p_4\in [1,\infty]$ such that $\frac{1}{p_1}+\frac{1}{p_2}=\frac{1}{p_3}+\frac{1}{p_4}=\frac{1}{p}$. Then, for all $\s>0$, we have
\begin{align}
\|uv\|_{B^{\s}_{p,q}} \les \|u\|_{B^{\s}_{p_1,q}}\|v\|_{L^{p_2}} + \|u\|_{L^{p_3}}\|v\|_{B^{\s}_{p_4,q}}. \label{Leibniz}
\end{align}
\textup{(iv)} \textup{(Sobolev embedding)} Let $1\leq p_2\leq p_1\leq \infty$, $q\in [1,\infty]$, and $s_2 \geq s_1+ 2(\frac{1}{p_2}-\frac{1}{p_1})$. Then, 
\begin{align}
\| u\|_{B^{s_1}_{p_1,q}} \les \|u\|_{B^{s_2}_{p_2,q}}. \label{Sobolev}
\end{align}
\end{lemma}

 \noi
Given two functions $f$ and $g$ on $\T^2$, we can expand the product $fg$ as
\begin{align}
fg & \hspace*{0.3mm}
= f\pl g + f \pe g + f \pg g\notag \\
& := \sum_{N < \frac{1}{2}M} P_{N} f \, P_M g
+ \sum_{\frac{1}{2}M\leq N \leq 2M} P_{N} f\,  P_M g
+ \sum_{M<\frac{1}{2}N} P_{N} f\,  P_M g.
\label{para1}
\end{align}
We will use the shorthand $N\ll M$ for $N<\frac{1}{2}M$ and $N\sim M$ for $\frac{1}{2}M \leq N \leq 2M$.

\begin{lemma}\label{LEM:para} 
Let $s_1, s_2 \in \R$ and $1 \leq p, p_1, p_2, q \leq \infty$ such that 
$\frac{1}{p} = \frac 1{p_1} + \frac 1{p_2}$.
Then, we have 
\begin{align}
\| f\pl g \|_{B^{s_2}_{p, q}} \les 
\|f \|_{L^{p_1}} 
\|  g \|_{B^{s_2}_{p_2, q}}.  
\label{para2a}
\end{align}

\noi
When $s_1 < 0$, we have
\begin{align}
\| f\pl g \|_{B^{s_1 + s_2}_{p, q}} \les 
\|f \|_{B^{s_1 }_{p_1, q}} 
\|  g \|_{B^{s_2}_{p_2, q}}.  
\label{para2}
\end{align}

\noi
When $s_1 + s_2 > 0$, we have
\begin{align}
\| f\pe g \|_{B^{s_1 + s_2}_{p, q}} \les 
\|f \|_{B^{s_1 }_{p_1, q}} 
\|  g \|_{B^{s_2}_{p_2, q}}  .
\label{para3}
\end{align}
In particular, if $s_1<0<s_2$ and $s_1+s_2>0$, the mapping $(f,g)\to fg$ extends to a continuous bilinear map from $B^{s_1}_{p_1,q}\times B^{s_2}_{p_2,q}$ to $B^{s_1+s_2}_{p,q}$.
\end{lemma}

\noi
See \cite{BCD} for proofs in the non-periodic case.
We then have some useful consequences.

\begin{lemma}
Let $0<\al<s$ and $0<\dl<2(1-s+\al)$. Then,
\begin{align}
\|fg\|_{H^{-s}} \les \|f\|_{H^{-s+\al}}\|g\|_{C^{s-\al+\dl}}, \label{prod1}\\
\|fg\|_{H^{-s+\al}} \les \|f\|_{H^{s-\al+\dl}}\|g\|_{C^{-s+\al}}. \label{prod2}
\end{align}
\end{lemma} 
\begin{proof}
First assume that $0<s<1$ and consider \eqref{prod1}.
Let $p=\frac{2}{1+s+\frac{1}{2}\dl}$ with $0<\frac{1}{2}\dl<1-s$ sufficiently small so that $p>1$.
By Sobolev embedding and Lemma~\ref{LEM:para}, we have
\begin{align}
\| fg\|_{H^{-s}} \les \|fg\|_{B^{\frac{1}{2}\dl}_{p,2}} \les \|f\|_{H^{-s+\al}}\|g\|_{B_{\infty,2}^{s-\al+\frac{1}{2}\dl}} \les  \|f\|_{H^{-s+\al}}\|g\|_{C^{s-\al+\dl}}. \label{prod3}
\end{align}
For \eqref{prod2},  with $p=\frac{2}{1+\frac{1}{2}\dl+s-\al}$ and $0<\dl<2(1-s+\al)$, when $0<s<1$:
\begin{align*}
\|fg\|_{H^{-s+\al}}\les \|fg\|_{B^{\frac 12 \dl}_{p,2}} \les \|f\|_{H^{s-\al+\dl}} 
\|g\|_{B^{-s+\al-\frac 12 \dl}_{\infty,2}} \les \|f\|_{H^{s-\al+\dl}}\|g\|_{C^{-s+\al}}.
\end{align*}
If $s\geq 1$, then we instead first use the embedding $B^{\frac{1}{2}\dl}_{1,2}\subset H^{-s}$ for both \eqref{prod1} and \eqref{prod2}.
\end{proof}

\subsection{A commutator estimate}

Our main goal in this subsection is to give a proof of the following commutator estimate. 

\begin{proposition}\label{PROP:com}
For any $s>0$ and $0<\eps<\frac{1}{4}\min(s,1)$, it holds that 
\begin{align}
\| \jb{\nb}^{s}(u^3)-3u^2 \jb{\nb}^{s} u\|_{L^{2}(\T^2)} \les \|u\|_{C^{s-\eps}(\T^2)}^{3}. \label{comL2}
\end{align}
\end{proposition}

Whilst the estimate \eqref{comL2} is far from optimal, it will be sufficient for our purposes as we only require the regularities on the right-hand side of \eqref{comL2} to be strictly less than $s$.
We need two preliminary results before we can give a proof of \eqref{comL2}.

\begin{lemma}
For any $s>0$ and $0<\eps<\frac12 s$, it holds that 
\begin{align}
\| u^3 - 3(u^2 \pl u)\|_{H^{s}(\T^2)}\les \|u\|_{C^{s-\eps}(\T^2)}^{3}. \label{u3com}
\end{align}
\end{lemma}
\begin{proof}
We have
\begin{align*}
u^3 = \sum_{N}  \big\{(P_{\leq 2N}u)^3-(P_{\leq N}u)^3  \big\}=  \sum_{N} P_{N}u \big[ (P_{\leq 2N}u)^2+(P_{\leq 2N}u)(P_{\leq N}u)+(P_{\leq N}u)^2\big],
\end{align*}
so that 
\begin{align}
u^3 - 3(u^2 \pl u) & = \sum_{N} P_{N}u \big[ (P_{\leq 2N}u)^{2}-P_{\leq \frac{1}{2}N}(u^2)\big] \label{I} \\
&\hphantom{XX}+ \sum_{N} P_{N}u \big[ P_{\leq 2N}u \cdot P_{\leq N}u -P_{\leq \frac{1}{2}N}(u^2) \big]  \label{II}\\
& \hphantom{XX} + \sum_{N} P_{N}u  \big[(P_{\leq N}u)^2 - P_{\leq \frac{1}{2}N}(u^2)\big]. \label{III}
\end{align}
We provide details only for estimating \eqref{I} as the \eqref{II} and \eqref{III} follow from similar ideas exploiting the existence of at least two similar and large frequencies.
We write 
\begin{align}
\begin{split}
(P_{\leq 2N}u)^{2}-P_{\leq \frac{1}{2}N}(u^2)&=\sum_{N_1,N_2} P_{>\frac{1}{2}N}[ P_{\leq 2N}P_{N_1}u \cdot P_{\leq 2N}P_{N_2}u] \\
& \hphantom{XXXX}  + P_{\leq \frac{1}{2}N} \big[ P_{\leq 2N}P_{N_1}u\cdot P_{\leq 2N}P_{N_2}u - P_{N_1}u \cdot P_{N_2}u\big].
\end{split} \label{comu21}
\end{align}
Let $M\in 2^{\N_0}$. For the first piece on the right-hand side of \eqref{comu21}, we have 
\begin{align}
\begin{split}
\bigg\| \sum_{N,N_1,N_2} &P_{M} \big[ P_{N}u \cdot P_{>\frac{1}{2}N}[ P_{\leq 2N}P_{N_1}u \cdot P_{\leq 2N}P_{N_2}u]\big]\bigg\|_{L^2} \\
& \les \sum_{\substack{N,N_1,N_2\\ N_1 \sim N \ges M}} \| P_{N}u \cdot P_{>\frac{1}{2}N}[ P_{\leq 2N}P_{N_1}u \cdot P_{\leq 2N}P_{N_2}u]\big]\|_{L^2} \\
& \les \sum_{\substack{N,N_1,N_2\\ N_1 \sim N \ges M}} \| P_{N}u\|_{L^{\infty}} \|P_{N_1}u\|_{L^{2}}\|P_{N_2}u\|_{L^{\infty}} \\
& \les \bigg( \sum_{\substack{N,N_1,N_2\\ N_1 \sim N \ges M}} (NN_1 N_2)^{-(s-\eps)} \bigg) \|u\|^{3}_{C^{s-\eps}} \\
& \les M^{-2(s-\eps)}\|u\|^{3}_{C^{s-\eps}}.
\end{split} \label{comu22}
\end{align} 
Using the characterisation 
\begin{align*}
\| f\|_{H^{s}}^2  \sim \sum_{M} M^{2s} \|P_{M}f\|_{L^{2}}^{2},
\end{align*}
we see that this contribution is controlled provided that $2s-4(s-\eps)<0$, enforcing $\eps<\frac{1}{2}s$. 
For the second piece on the right hand side of \eqref{comu21}, we further write it as 
\begin{align*}
 P_{\leq \frac{1}{2}N} &\big[ P_{\leq 2N}P_{N_1}u\cdot P_{\leq 2N}P_{N_2}u - P_{N_1}u \cdot P_{N_2}u\big] \\
 & =  P_{\leq \frac{1}{2}N} \big[ P_{\leq 2N}P_{N_1}u \cdot P_{>2N}P_{N_2}u - (P_{>2N}P_{N_1}u)(P_{N_2}u)\big]
\end{align*}
Again, we have at least that $\max(N_1,N_2)\ges N$ in each of these terms, and we can proceed similar to \eqref{comu22} provided that $\eps<\frac{1}{2}s$. This completes the proof.
\end{proof}

\begin{lemma}\label{LEM:com2}
Given $s>0$ and $0<\eps<\frac{1}{2}s$, there exists a multilinear operator $R$ such that 
\begin{align}
|\nb|^{s}(f \pl g) - f \pl (|\nb|^{s}g) = s \sum_{j=1}^{2} (\partial_{j}f) \pl (|\nb|^{s-2}\partial_j g) + R(f,g),
\end{align}
and $R$ satisfies 
\begin{align*}
\| R(f,g)\|_{L^{2}} \les \|f\|_{C^{s-\eps}}\|g\|_{B^{s-\eps}_{2,\infty}}.
\end{align*}
\end{lemma}
 Lemma~\ref{LEM:com2} is well known in the case of $\R^2$; see the more general result in \cite[Theorem 2.92]{BCD}. For an argument in the periodic setting, we refer to \cite[Lemma 13]{STX}.

\begin{proof}[Proof of Proposition~\ref{PROP:com}]
We begin by writing 
\begin{align}
\jb{\nb}^{s}(u^3) -3u^2 \jb{\nb}^{s} u    =& |\nb|^{s}(u^3) -3u^2 |\nb|^{s} u  \label{commain} \\
& + (\jb{\nb}^{s}-|\nb|^{s})(u^3) - 3u^2 ( \jb{\nb}^{s}-|\nb|^{s})(u). \label{comdiff}
\end{align}
By the mean-value theorem,
\begin{align}
|\jb{n}^{s}-|n|^{s}| \les \jb{n}^{s-1} \quad \text{for all} \quad n\in \Z^{2}. \label{multiplierdiff}
\end{align}
If $s\leq 1$, then \eqref{multiplierdiff} implies 
\begin{align}
\|(\jb{\nb}^{s}-|\nb|^{s})(u^3) \|_{L^2} \les \|u\|_{L^6}^{3}. \label{comdiff1}
\end{align}
Otherwise, if $s>1$, \eqref{Leibniz} implies
\begin{align}
\|(\jb{\nb}^{s}-|\nb|^{s})(u^3) \|_{L^2}  \les \| u^3\|_{H^{s-1}}\les  \|u\|_{L^{\infty}}^{2} \|u\|_{H^{s-1+\eps}}\label{comdiff2}
\end{align}
for any $\eps>0$.
We also have by \eqref{multiplierdiff} that
\begin{align}
\| u^2 ( \jb{\nb}^{s}-|\nb|^{s})(u)\|_{L^2} \les \|u\|_{L^{\infty}}^{2}\|u\|_{H^{s-1}}. \label{comdiff3}
\end{align}
Now \eqref{comdiff1}, \eqref{comdiff2}, and \eqref{comdiff3} give 
\begin{align*}
\eqref{comdiff} \les \|u\|_{C^{s-\eps}}^{3}.
\end{align*}
For \eqref{commain}, we write
\begin{align}
\eqref{commain}  = & |\nb|^{s}(u^3) -3|\nb|^{s}(u^{2} \pl u)  \label{com1}\\
& + 3|\nb|^{s}(u^2 \pl u) - 3 u^2 \pl |\nb|^{s}u \label{com2} \\
& +3u^2 \pl |\nb|^{s} u - 3u^2 |\nb|^s u. \label{com3}
\end{align}
First, we have 
\begin{align*}
\eqref{com3} = -3\big( u^2 \pe |\nb|^{s} u + u^2 \pg |\nb|^s u \big).
\end{align*}
Using \eqref{para3} and \eqref{algebra}, we have 
\begin{align*}
\| u^2 \pe |\nb|^{s} u  \|_{L^2}\les \| u^2 \pe |\nb|^{s} u \|_{C^{\frac{s}{4}}} \les \| u^2\|_{C^{\frac{s}{2}}} \||\nb|^{s}u \|_{C^{-\frac{s}{4}}} \les \|u\|^{2}_{C^{\frac{s}{2}}} \| u\|_{C^{\frac{3s}{4}}} \les \|u\|_{C^{s-\eps}}^3.
\end{align*}
 We obtain the same bound for $u^2 \pg |\nb|^s u$ by using \eqref{para2}. Thus, 
 \begin{align*}
\eqref{com3} \les \|u\|_{C^{s-\eps}}^{3}.
\end{align*}
For \eqref{com1}, we use  \eqref{u3com}.
It remains to control \eqref{com2} for which we use Lemma~\ref{LEM:com2}. 
For the remainder term, we obtain the bound $\les  \|u\|_{C^{s-\eps}}^{3}$, after using \eqref{algebra}.
For the first term, we fix $j=1,2$, and use \eqref{para2a} and \eqref{algebra}, 
\begin{align*}
\| (\partial_j u^2)\pl (|\nb|^{s-2} \partial_j u)\|_{L^2}  & \leq \sum_{N_1 \ll N_2} \| (\partial_j P_{N_1}(u^2))\pl (|\nb|^{s-2} \partial_j P_{N_2}u)\|_{L^2} \\
& \les \sum_{N_1 \ll N_2} \| \partial_j P_{N_1}(u^2)\|_{L^{\infty}_{x}}\| |\nb|^{s-2} \partial_j P_{N_2}u\|_{H^{s}} \\
& \les \sum_{N_1 \ll N_2} N_1^{1-s+\eps} N_{2}^{s-1}N_{2}^{-s+\eps} \| u^2\|_{C^{s-\eps}} \|u\|_{H^{s-\eps}}.
\end{align*}
Now, if $s>1$, taking $\eps<\frac{1}{2}$ we may sum over the dyadics. If $s\leq 1$, we need $\eps<\frac{1}{2}s$. 
\end{proof}

%
%
%

\subsection{The white noise and stochastic convolution}

The space-time white noise $\xi$ is the unique (in law) random distribution such that $\{ \jb{\xi,\phi} \}_{\phi\in L^2(\R\times \T^2)}$ is a family of centred Gaussian random variables on a probability space $(\O, \F, \prob)$ with the property that 
\begin{align}
\E\big[ \jb{\xi, \phi} \jb{\xi, \psi}  \big]=\jb{\phi, \psi}_{L^2(\R\times \T^2)}  \label{whiteprop}
\end{align}
for any $\phi, \psi\in C^\infty_c(\R\times \T^2)$. 
For $t\geq 0$, we set $$\tilde{\F}_{t}=\s\big( \big\{ \jb{\xi,\phi} \, : \, \phi\vert_{(t,+\infty)\times \T^2}\equiv 0, \, \phi\in L^2(\R\times \T^2)\big\}\big)$$ and we denote by $\{\F_{t}\}_{t\geq 0}$ the usual augmentation of the filtration $\{ \tilde{\F}_{t}\}_{t\geq 0}$ (see~\cite[p. 45]{RevuzYor}).\label{filtration_def}


We define the stochastic convolution $\stick_{t}(\xi)$ in a pathwise manner as the space-time distribution satisfying
\begin{align} \label{sticktested}
\jb{\stick_t(\xi), \phi}_{L^{2}_{t',x}} := {\sqrt{2}} \jb{ \xi(t'), \ind_{\{ 0\leq t'\leq t\}}\jb{\nb}^{-s}\pi_2 S(t-t')^{\ast}\phi}_{L^{2}_{t',x}},
\end{align}
for all test functions $\phi =( \phi_1 , \phi_2)^{\top}:\T^2 \to \R$.  Here, 
$S(t)$ is the linear propagator for the (damped) wave equation, and it is given by the formula
\begin{equation} \label{linsol}
e^{-\frac t2}
\begin{pmatrix}\cos\big(t\snb\big) + \frac12\frac{\sin(t\snb)}{\snb}
&\frac{\sin(t\jb{\snb})}{\snb} \\
-\big(\snb+\frac1{4\snb}\big)\sin(t\snb) & 
\cos(t\snb) -\frac12\frac{\sin(t\snb)}{\snb}
\end{pmatrix},
\end{equation}
where $\snb:= (\frac 34 -\Dl)^{\frac 12}$.
For smooth $\xi$, the definition \eqref{linsol} corresponds exactly to 
\begin{align}
\stick_t(\xi) =  \int_0^t S(t-t')\vect{0}{\sqrt{2}\jb{\nb}^{-s}\xi(t')} dt'.
\label{stick}
\end{align}
It enjoys the following regularity property. 


\begin{lemma} \label{PROP:stickZ}
Let $0<\al<s$ and $N\in 2^{\N_0}\cup\{\infty\}$. 
Then, $\Pi_{\leq N}\stick_{t}(\xi)\in C([0,+\infty); Z^{\al})$
 almost surely. 
 Moreover, for every $T>0$, 
\begin{align}
\E\bigg[  \sup_{N\in 2^{\N_0}\cup \{\infty\}}  \| \Pi_{\leq N}\stick_{t}(\xi)\|_{C([0,T];Z^{\al})}^{p}\bigg]
\leq C(p,T), \label{momentsstick}
\end{align}
where $1\leq p<\infty$.
\end{lemma}

Before we give the proof, we first recall a quantitative version of the Kolmogorov continuity criterion, which is a slight modification of \cite[Lemma 2.4]{KMV}.

\begin{lemma}\label{LEM:Kolm}
Given $\eps>0$, $j\in \N_0$, $\g \in (0,1)$, $1\leq p<\infty$, a Banach space $X$, and a process $F:[j,j+1]\to X$ that is almost surely continuous,
\begin{align}
\E\big[ \|F\|_{C^{\g}([j,j+1];X)}^{p} \big] \les_{\g,\eps} \E[ \| F(j)\|_{X}^{p}] + \sup_{j\leq t'<t\leq j+1} \E \bigg[ \frac{ \|F(t)-F(t')\|_{X}^{p} }{|t-t'|^{1+\g p +\eps}} \bigg]. \label{Kolm}
\end{align}
\end{lemma}

\begin{proof}[Proof of Lemma~\ref{PROP:stickZ}]
The almost sure properties and the moment bound \eqref{momentsstick} when $N=+\infty$ were proved in \cite[Proposition 2.7 and 2.8]{FT1}. 
Note that this does not automatically imply \eqref{momentsstick}, since the operator $\Pi_{\leq N}$ is not bounded on $\CC^{\al}$.
However, the estimate immediately follows form the arguments in \cite[Proposition 2.8]{FT1}, so here we highlight the main modifications to the argument.
For \eqref{momentsstick}, a telescoping series argument and Minkowski's inequality shows that it suffices to prove that
\begin{align}
\E\big[ \| \Pi_{\leq 1}\stick_{t}(\xi)\|_{C([0,T];Z^{\al})}^{p}\big] & \les_{p,T} 1,\label{diffstick2}\\
\E \big[  \| \Pi_{\leq 2M}\stick_{t}(\xi)-\Pi_{\leq M}\stick_{t}(\xi)\|_{C([0,T];Z^{\al})}^{p}] &\les_{p,T} M^{-\ta} \label{diffstick1}
\end{align}
for some $\ta>0$, and every $M\in 2^{\N_0}$. 

 For shorthand, we define $\stick_{t}(\xi_M):= \Pi_{\leq 2M}\stick_{t}(\xi)-\Pi_{\leq M}\stick_{t}(\xi)$.  To control the supremum over $t\in [0,T]$ in \eqref{diffstick1}, we decompose $[0,T]$ into unit intervals and then apply Lemma~\ref{LEM:Kolm} on each interval. We only discuss the first term on the right-hand side of \eqref{Kolm}. The estimates required for the time-differences term in \eqref{Kolm} follow similarly using the strategy below and the estimates in the proof of \cite[Proposition 2.7]{FT1}, up to the modification we give below. Fix $t\in (0,T]$. 
Recalling the definition of the $Z^{\al}$-norm in \eqref{Znorm}, for finite $p>\frac{2}{s-\al}$, Minkowski's inequality implies
\begin{align*}
\E[ \| \stick_{t}(\xi_M)\|_{Z^{\al}}^{p}]^{\frac{1}{p}} & \les \sum_{T'\in \N} e^{\frac{T'}{8}} 
\E \bigg[  \sup_{\tau \in [T'-1,T']} \| S(\tau)\stick_{t}(\xi_M)\|_{\CC^{\al}}^{p} \bigg]^{\frac 1p}.
\end{align*}
To control the supremum over $\tau$, we again need to use Lemma~\ref{LEM:Kolm}. We again only discuss the contribution from the first term from \eqref{Kolm}. By Sobolev embedding and the Gaussianity of $\xi$, we have
\begin{align*}
\E\big[ \| \jb{\nb}^{\al+\frac{2}{p}} S(T'-1)\stick_{t}(\xi_M)\|_{\W^{0,p}}^{p} \big]^{\frac{1}{p}}
& \les \bigg(  \sum_{n} \jb{n}^{2\al+\frac{4}{p}} \, \E[ |\jb{ S(T'-1)\stick_{t}(\xi_M), \phi_{n}}|^{2}] \bigg)^{\frac 12},
\end{align*}
where $\phi_{n} = ( e_n, \jb{n}^{-1}e_n)^{\top}$. Using \eqref{sticktested} and the frequency localisations from the difference in $\stick_{t}(\xi_M)$, which forces $|n|_{\infty}\sim M$, we bound this by 
\begin{align*}
e^{-\frac{T'}{2}}\bigg( \sum_{|n|_{\infty}\sim M} \jb{n}^{2\al+\frac{4}{p}-2s-2} \bigg)^{\frac 12} \les e^{-\frac{T'}{2}} M^{-(s-\al-\frac{2}{p})},
\end{align*}
which gives \eqref{diffstick1}, as long as $p>\frac{2}{s-\al}$. Finally, \eqref{diffstick2} follows similarly. 
\end{proof}

\section{The flow for \eqref{SDNLW}}

\subsection{Review of well-posedness}

In this section, we recall elements from \cite{FT1} of the notion of solution and associated well-posedness for the equation \eqref{SDNLW}.
We view \eqref{SDNLW} in the following integral formulation: 
\begin{align}
\u= S(t)\u_0 +\int_{0}^{t}S(t-t') \begin{pmatrix} 0 \\ {\text{\small{$\sqrt{2}$}}} \jb{\nb}^{-s}\xi(t') \end{pmatrix} dt' -\int_{0}^{t}S(t-t')\begin{pmatrix} 0 \\ (\pi_{1}\u(t'))^{3} \end{pmatrix}dt',
\label{duhamel}
\end{align}
where $\pi_1$ denotes the projection to the first coordinate.
Motivated by \eqref{duhamel} and these considerations about the stochastic convolution~\eqref{stick}, we define a solution of $\eqref{SDNLW}$ in terms of the first-order expansion~\cite{mckean, Bo96, dpd}:
\begin{align} \label{soldef}
\u(t) = \Phi_t(\u_0,\xi) =S(t)\u_0 +\stick_{t}(\xi)+\w(t),
\end{align}
where $\w=(w, \dt w)^{\top}$ solves the equation
\begin{align}
\w(t) = - \int_0^t S(t-t') \vect{0}{\pi_1( S(t')\u_0+\stick_{t'}(\xi)+\w(t'))^{3}  } dt',
\label{veqn}
\end{align}
We note that the equation \eqref{veqn} is chosen to be the mild formulation of the differential equation 
$$ \dt^{2}w + \dt w+ (1-\Delta) w + \pi_1(S(t)\u_0 +\stick_{t}(\xi)+\w(t))^3 = 0, $$
which we obtain by inserting the ansatz \eqref{soldef} into the equation \eqref{SDNLW}. We recall the well-posedness statement for \eqref{veqn} below in Proposition~\ref{PROP:LWP}.

In order to make the arguments in this paper rigorous, we require approximating finite dimensional versions of \eqref{SDNLW}, which we describe now.
Given $N\in 2^{\N_0}$, we define the truncated system of \eqref{NLWvec} by 
\begin{equation}
\begin{cases}
\partial_t \vect{u^N}{v^N}  &= -\begin{pmatrix} 0 & -1 \\  1-\Delta & 1 \end{pmatrix} \begin{pmatrix} u^N \\ v^N \end{pmatrix}  - \Pi_{\leq N} \vect{0}{(\Pi_{\leq N}u^N)^3} + \vect{0}{\sqrt{2} \jb{\nb}^{-s}\xi} \\
 \begin{pmatrix} u^N \\ v^N \end{pmatrix}(0) & = \begin{pmatrix} u_0 \\ u_1 \end{pmatrix}.
\end{cases} \label{NLWvectrunc}
\end{equation}
Solutions to \eqref{NLWvectrunc} naturally decompose as:
\begin{align*}
\u^{N}(t) = \Pi_{\leq N} \u^N (t)  + \Pi_{>N}\u^N(t)= \u_{N} (t)+\u^{\perp}_{N}(t),
\end{align*}
where the low-frequency part $\u_N (t) =: (u_N, v_N)$ solves 
\begin{equation}
\begin{cases}
\partial_t \vect{u_N}{v_N}  &= -\begin{pmatrix} 0 & -1 \\  1-\Delta & 1 \end{pmatrix} \begin{pmatrix} u_N \\ v_N \end{pmatrix}  - \Pi_{\leq N} \vect{0}{u_N^3} + \vect{0}{\sqrt{2} \jb{\nb}^{-s} \Pi_{\leq N}\xi} \\
 \begin{pmatrix} u_N \\ v_N \end{pmatrix}(0) & = \begin{pmatrix}  \Pi_{\leq N} u_0 \\ \Pi_{\leq N}u_1 \end{pmatrix}.
\end{cases} \label{NLWuN}
\end{equation}
and the high-frequency part $\u^{\perp}_{N} = :(u_{N}^{\perp}, v_{N}^{\perp})$ solves the linear damped stochastic wave equation:
\begin{equation}
\begin{cases}
\partial_t \vect{u^{\perp}_N}{v^{\perp}_N}  &= -\begin{pmatrix} 0 & -1 \\  1-\Delta & 1 \end{pmatrix} \begin{pmatrix} u^{\perp}_N \\ v^{\perp}_N \end{pmatrix} + \vect{0}{\sqrt{2} \jb{\nb}^{-s} \Pi_{>N}\xi} \\
 \begin{pmatrix} u^{\perp}_N \\ v^{\perp}_N \end{pmatrix}(0) & = \begin{pmatrix}  \Pi_{> N} u_0 \\ \Pi_{>N}u_1 \end{pmatrix}.
\end{cases} \label{NLWuNperp}
\end{equation}


We write solutions $\u_{N}$ to \eqref{NLWuN} as 
\begin{align}
\begin{split}
\u_N (t) &=S(t) \Pi_{\leq N}\u_0 +\stick_{t}(\Pi_{\leq N}\xi)+\w_N (t) \\
&  = : \Phi_{t}^{\text{lin}}(\Pi_{\leq N}\u_0, \Pi_{\leq N}\xi)+ \w_{N}(t).
\end{split}\label{uN}
\end{align}
where 
\begin{align}
\Phi_{t}^{\text{lin}}(\u_0,\xi) : = S(t) \u_0 + \stick_{t}(\xi) \label{Philin}
\end{align}
 is the solution map for the linear stochastic damped NLW,
 and $\w_N$ solves 
\begin{align}
\w_{N}(t) = - \int_0^t S(t-t')\Pi_{\leq N} \vect{0}{(\pi_1 \Pi_{\leq N}(S(t')\u_0+\stick_{t'}(\xi)+\w_{N}(t')))^3  } dt'. 
\label{veqnN}
\end{align}
We will occasionally use the fact that 
$\Pi_{\leq N}\Phi^{\text{lin}}_{t}(\u_0,\xi) = \Phi^{\text{lin}}_{t}(\Pi_{\leq N}\u_0,\Pi_{\leq N}\xi)$ without explicit reference. 
We have the following global well-posedness results for~\eqref{SDNLW} and \eqref{veqnN}.

\begin{proposition}\label{PROP:LWP}\cite[Section 3]{FT1}
Let $N\in 2^{\N_0}$, $s>0$, $0<\al<s$, $T>0$, $\u_0 \in \X$, and $\w_0\in \H^1$. Then, the equations 
\begin{align}
\w(t) &=S(t)\w_0 - \int_0^t S(t-t')\vect{0}{(\pi_1 (S(t')\u_0+\stick_{t'}(\xi)+\v(t')))^3  } dt', \label{veqnd}\\
\w_{N}(t) &=S(t)\w_0 - \int_0^t S(t-t') \Pi_{\leq N} \vect{0}{(\pi_1 \Pi_{\leq N}(S(t')\u_0+\stick_{t'}(\xi)+\v_{N}(t')))^3 } dt',
\label{veqnNd}
\end{align}
are almost surely globally-well posed in $\H^{1}(\T^2)$. Moreover, the following hold true:

\smallskip
\noi
\textup{(i)} With $\w_N$ denoting the solution to \eqref{veqnNd} with $\w_0\equiv 0$,
 for every $N\in 2^{\N_0}\cup\{\infty\}$ and for $0<\al<s\wedge 1$, 
it holds that
\begin{align}
\begin{split}
\|\w_N(t)\|_{C([0,T];\H^1)}^{2} &\leq C( \|\Pi_{\leq N}\u_0\|_{X^{\al}}, \|\Pi_{\leq N}\stick_{t}(\xi)\|_{C([0,T];X^{\al})}, T).
\end{split} \label{wNbd}
\end{align}


%

\smallskip 
\noi
\textup{(ii)}
For every $T<\infty$, we have 
\begin{align*}
\lim_{N\to \infty} \|\w-\w_N\|_{C([0,T];\H^1)}=0 \quad\textup{a.s.}
\end{align*}
\end{proposition}

The existence and uniqueness of $\w$ and $\w_N$ is guaranteed by an application of Banach fixed point theorem, which uses only the fractional Leibniz rule and Sobolev embedding. The global bound \eqref{wNbd} follows from a Gronwall argument on an energy quantity resembling that of the Hamiltonian for the undamped, non-stochastic NLW; see \cite[(3.12)]{FT1}. The a priori bound \eqref{wNbd} is only of interest when $s<1$. When $s\geq 1$, we still have global well-posedness for $\w$ by a Gronwall argument with the energy 
\begin{align}
\begin{split}
H_{N}(u,v)&:= H(\Pi_{\leq N} u, \Pi_{\leq N}v),\\
H(u,v) &:= \frac{1}{2}\int_{\T^{2}}u^2+| \nb u|^{2} +v^2 dx+ \frac{1}{4}\int_{\T^2} u^4 dx.
\end{split} \label{HN}
\end{align}
We omit the details. See for example \cite{BuTz2}.


We remark that by regularity counting arguments, we expect $\w$ to be more regular than $\H^{1}(\T^2)$ and to belong to $\H^{1+\al}(\T^2)$. This additional regularity can be justified a posteriori from Proposition~\ref{PROP:LWP}, and we will make use of it later; see Lemma~\ref{LEM:Qscontrol}.

Recall that  for$\u_0\in \X$, and $N\in 2^{\N_0}$, we have
\begin{align}
\Phi_{t}^{N}(\u_0; \xi):=\u^{N}(t) = \u_{N}(t) + \u^{\perp}_{N}(t), \label{PhiN}
\end{align}
where $\u_N$ has the decomposition \eqref{uN} and $\u^{\perp}_{N}(t)$ solves the linear equation \eqref{NLWuNperp}.
When $N=\infty$, we write 
\begin{align}
\Phi^{\infty}_{t}(\u_0;\xi):=\Phi_{t}(\u_0;\xi):=S(t)\u_0  +\stick_{t}(\xi) +\w(t), 
\label{PhiN2}
\end{align}
where $\w$ solves \eqref{veqnd}. We have the following convergence statement for the nonlinear flows. 

\begin{lemma}\label{LEM:ASconvflow}
Let $s>0$, $0<\al<s$, $\u_0\in X^{\al}$, and $T>0$. Then, $\Phi_{t}^{N}(\u_0,\xi)$ converges to $\Phi_{t}(\u_0,\xi)$ almost surely in $C([0,T];\H^{\al}(\T^2))$ as $N\to \infty$. 
\end{lemma}
\begin{proof}
When $\al<1$, this is immediate from \eqref{PhiN}, \eqref{PhiN2}, and Proposition~\ref{PROP:LWP} (iii). When $\al \ge 1$, we can estimate the difference $\w_N -\w$ in $\H^{1+\al}$ directly using the Duhamel formulation and exploiting the convergence in Proposition~\ref{PROP:LWP} and the convergence of $\Pi_{\leq N}S(t)\u_0$ and $\Pi_{\leq N}\stick_{t}(\xi)$ in $\H^{\al}$.
\end{proof}

\subsection{Markov semigroups}

In this section, we recall how the flows \eqref{PhiN} and \eqref{PhiN2} generate Markov semigroups on $X^{\al}$.
We denote by $\B_{b}(\X)$ the space of measurable, bounded functions from $\X$ to $\R$, and by $\mathcal{C}_{b}(\X)$ the space of bounded, continuous functions from $\X$ to $\R$, both endowed with the norm 
\begin{align*}
\|F\|_{L^{\infty}}:=\sup_{\u_0 \in \X}|F(\u_0)|.
\end{align*}
Given $N\in 2^{\N_0}\cup\{\infty\}$ and $F\in \mathcal{B}_{b}(\X)$, we define the family of bounded linear operators $\Pt{t}^{N}:F\mapsto \Pt{t}^NF$, by 
\begin{align}
\Pt{t}^NF(\u_0):=\E[ F(\Phi_{t}^N(\u_0;\xi))], \quad t\geq 0, \label{Ptdefn}
\end{align}
for every $\u_0\in \X$, and $\Pt{t}:=\Pt{t}^{\infty}$.
We recall from \cite[Section 4]{FT1}, that $\Pt{t}$ defines a Markov semigroup. The same is true for the operators $\Pt{t}^{N}$, $N\in 2^{\N_0}$, defined using the truncated flow. 
We state the main properties of the semigroups $\{\Pt{t}^{N}\}_{N}$. For the proof, see \cite[Proposition 4.1, Lemma 4.2, and Proposition 4.3]{FT1}. 

\begin{lemma} \label{LEM:Markov}
Let $N\in 2^{\N_0}\cup \{\infty\}$. Then, the following hold true: 

\smallskip
\noi
\textup{(i)} Let $\u_0\in \X$. Then the map $[0,\infty)\to \X$ defined by $t\to \Phi_{t}^{N}(\u_0;\xi)$ is adapted with respect to the filtration $\{\mathcal{F}_{t}\}_{t\geq 0}$. In particular, for every $t\geq 0$,
\begin{align}
\Phi^{N}_{t}(\u_0; \xi) = \Phi^{N}_{t}(\u_0 , \ind_{[0,t]}\xi) \label{adapt}
\end{align}

\smallskip
\noi
\textup{(ii)}For a.e.\ realisation of $\xi$, and for every $t,h\geq 0$, we have that 
\begin{equation}\label{semigroup}
\Phi^{N}_{t+h}(\u_0, \xi) = \Phi^{N}_{h}(\Phi^{N}_{t}(\u_0, \ind_{[0,t]}\xi), R_{{t}}(\ind_{[t,\infty)}\xi)),
\end{equation}
where $R_{t}$ is the (right) time translation by $t_0$, which is defined on space-time distributions as an extension of the map $R_{t}f(s,x) = f(s+t,x),$ for $f\in L^1_{\mathrm{loc}}(\R \times \T^2)$.

\smallskip
\noi
\textup{(iii)} $\{\Phi_{t}^{N}\}_{t\geq 0}$ is a Markov process and $\{\Pt{t}^{N}\}_{t\geq 0}$ is a Markov semigroup.

\end{lemma}

In fact, more than Lemma~\ref{LEM:Markov} (iii) is true; namely, the maps $\{\Phi_{t}^{N}\}_{t\geq 0}$ satisfy the strong Markov property (see \cite[Section 4]{FT1}), but we will not need this extra information in this article.

\section{Modified measures and sample path properties}
\label{SEC:measures}

\subsection{Modified measures}

The Gaussian measure $\vec{\mu}_s$ in \eqref{mus} can be formally written as 
\begin{align*}
d\vec{\mu}_{s}= Z_{s}^{-1} e^{-G_{s}(u,v)} du dv,
\end{align*}
where 
\begin{align}
G_{s}(u,v) : = \frac{1}{2}\int_{\T^2}(\jb{\nb}^{s}v)^2  + (\jb{\nb}^{s+1}u)^2   dx. \label{Gs}
\end{align}
Rigorously, we define $\vec{\mu}_s :=\text{Law}(\u^{\o})= \text{Law}((u^{\o},v^{\o}))$, where $(u^{\o}, v^{\o})$ denote the following pair of random Fourier series:
\begin{align*}
u^{\o}(x) & = \sum_{n\in \Z^2} \frac{g_{n}(\o)}{\jb{n}^{2s+2}}e^{in\cdot x} \quad \text{and} \quad  v^{\o}(x) = \sum_{n\in \Z^2} \frac{h_{n}(\o)}{\jb{n}^{2s}}e^{in\cdot x},
\end{align*}
where $\{g_n\}_{n\in \Z^2}$ and $\{h_n\}_{n\in \Z^2}$ are families of standard complex-valued Gaussian random variables conditioned so that $g_{n}=\cj{g_{-n}}$ and $h_n=\cj{h_{-n}}$. More precisely, we define the index set 
\begin{align*}
\Ld : = (\Z \times \Z_{+})\cup (\Z_{+}\times \{0\}) \cup \{(0,0)\},
\end{align*}
and let $\{g_n,h_n\}_{n\in \Ld}$ be standard complex valued Gaussian random variables, with $g_0,h_0$ real-valued, and then define $g_{-n}= \cj{g_n}$, $h_{-n}=\cj{h_n}$ for $n\in \Z^2$.

Given $N\in 2^{\N_0}$, we define the marginals
\begin{align}
\vec{\mu}_{s,N}&:=(\Pi_{\leq N})_{\ast}\vec{\mu}_s =\text{Law}(( \Pi_{\leq N}u^{\o}, \Pi_{\leq N}v^{\o})),\\
\vec{\mu}_{s,N}^{\perp}&:=(\Pi_{> N})_{\ast}\vec{\mu}_s =\text{Law}(( \Pi_{> N}u^{\o}, \Pi_{> N}v^{\o})).
\end{align}
 By independence, we then have the decomposition
\begin{align}
d\vec{\mu}_{s} = d\vec{\mu}_{s,N} \otimes d\vec{\mu}_{s,N}^{\perp}. \label{mudecomp}
\end{align}
We let $\mathcal{V}_{N}$ be the real vector space
\begin{align*}
\mathcal{V}_{N} = \text{span}\{ 1, \cos(n\cdot x) , \sin(n\cdot x):\, n\in \Ld_{N}^{\ast}\},
\end{align*}
where $\Ld_{N}^{\ast}: = \{n\in \Z^2\,:\, 0<|n|\leq N\}\cap \Ld$, and put the usual scalar product on $\mathcal{V}_{N}$. We endow $\mathcal{V}_N \times \mathcal{V}_N$ with a Lebesgue measure $L_N$ as follows: 
with 
\begin{align*}
\Pi_{\leq N} u (x) = u_{N}(x)  =\sum_{|n|\leq N} \ft u_n e^{in\cdot x}, \quad \ft u_{-n} = \cj{\ft u_n},
\end{align*}
we put $a_0=\ft{u}(0)$, $a_{n,1}= \text{Re}\, \ft u_n$, and $a_{n,2}=\text{Im}\, \ft u_n$ so that 
\begin{align}
u_{N}(x) = a_0 +\sum_{n\in \Ld_{N}^{\ast}} \{ a_{n,1}\, 2\cos(n\cdot x) + a_{n,2}(-2\sin(n\cdot x))\}. \label{uNdecomp}
\end{align}
We then define $L_N$ as the Lebesgue measure on $\mathcal{V}_N \times \mathcal{V}_N$ with respect to the orthogonal basis 
\begin{align*}
\big\{ 1, \{ 2 \cos(n\cdot x) , -2\sin(n\cdot x)\}_{n\in \Ld_{N}^{\ast}}\big\}\times \big\{ 1, \{ 2 \cos(n\cdot x) , -2\sin(n\cdot x)\}_{n\in \Ld_{N}^{\ast}}\big\}.
\end{align*}

For every $N\in \mathbb{N}$, the Gaussian measure $\vec{\mu}_{s,N}^{\perp}$ is invariant under the dynamics of the high-frequency equation \eqref{NLWuNperp}. Indeed, a tedious but straightforward computation shows that
\begin{align*}
\E[ |\mathcal{F}\{ \pi_1 \u_{N}^{\perp}\}(t,n)|^{2}] =\jb{n}^{-2s-2} \quad \text{and} \quad \E[ |\mathcal{F}\{ \pi_2 \u_{N}^{\perp}\}(t,n)|^{2}] =\jb{n}^{-2s}
\end{align*}
for all $t\in \R$, and since $\u_{N}^{\perp}$ evolves linearly, Gaussianity is preserved.

As discussed in \cite{OTz2, GOTW}, it is not clear how to study the quasi-invariance  of the measure $\vec{\mu}_{s}$ in \eqref{mus} under the nonlinear flow of \eqref{SDNLW} directly. Indeed, a renormalisation is needed to make sense of the time evolution of $G_{s}$ in \eqref{Gs}. Thus, following \cite{GOTW}, we define 
\begin{align}
\begin{split}
\s_{N} &: = \E_{\vec{\mu}_{s}} \bigg[ \int_{\T^2} (\jb{\nb}^{s}\Pi_{\leq N} u (x))^2 dx \bigg]  =  \sum_{|n|_{\infty}\leq N} \frac{1}{\jb{n}^{2}} \sim \log N. 
\end{split} \label{sigma1}
\end{align}
Then, we set 
\begin{align}
Q_{s,N}(u) : = (\jb{\nb}^{s} \Pi_{\leq N}u)^{2} -\s_{N}. \label{Qs}
\end{align}
We now define a modified version of \eqref{Gs} by 
\begin{align}
E_{s,N}(u,v) = G_{s}(\Pi_{\leq N}u, \Pi_{\leq N}v) + R_{s,N}(u),
\end{align}
where 
\begin{align}
R_{s,N}(u) : = \frac{3}{2}\int_{\T^2} Q_{s,N}(u)(\Pi_{\leq N}u)^{2}dx + \frac{1}{4}\int_{\T^2}(\Pi_{\leq N}u)^4 dx. \label{RsN}
\end{align}
It follows from Parseval's theorem that 
\begin{align}
E_{s,N}(u,v)  = \EE_{s,N}(u, v) + H_{N}(u,v), \label{EsN}
\end{align}
where
\begin{align}
\begin{split}
\EE_{s,N}(u, v) & = \frac{1}{2}\int_{\T^2} (\jb{\nb} m(\nb) \Pi_{\leq N}u)^2 +(m(\nb)\Pi_{\leq N}v)^{2} + \frac{3}{2}\int_{\T^2} Q_{s,N}(u)(\Pi_{\leq N}u)^{2}dx,
\end{split}\label{EEsN} 
\end{align}
where $m(\nb):= \sqrt{\jb{\nb}^{2s}-1}$,
and $H_{N}$ is the energy in \eqref{HN}.

The following lemma justifies the construction of the modified measures.

\begin{lemma}\label{LEM:nus}
Let $s>0$ and $1\leq p<\infty$. 
Then, there exists $R_{s}\in L^{p}(\vec{\mu}_{s})$ such that $R_{s,N}\to R_{s}$ in $L^{p}(\vec{\mu}_s)$ as $N\to \infty$.
Moreover, there exists $C_{p}>0$ such that 
\begin{align}
\sup_{N\in \N} \big\| e^{-R_{s,N}(u)}\big\|_{L^{p}(\vec{\mu}_{s})} \leq C_p <\infty. \label{unifbd}
\end{align}
Furthermore, 
\begin{align}
\lim_{N\to \infty} \big\| e^{-R_{s,N}(u,v)} - e^{-R_{s}(u,v)}\big\|_{L^{p}(\vec{\mu}_s)}=0, \label{eRLpconv}
\end{align}
and thus, for any $\s<s$, the restriction to $\H^{\s}(\T^2)$ of the measures 
\begin{align}
d\vec{\nu}_{s,N}:=Z_{s,N}^{-1} \, e^{-R_{s,N}(u)} d\vec{\mu}_{s} \label{nusN}
\end{align}
converge to a measure
\begin{align}
d\vec{\nu}_{s}= Z_{s}^{-1} e^{-R_{s}(u)} d\vec{\mu}_{s} \label{nus}
\end{align}
 in total variation, where $Z_{s,N}, Z_{s} \in (0,\infty)$ are defined so that $\vec{\nu}_{s,N}$ and $\vec{\nu}_{s}$ are probability measures.
 In particular, $\vec{\nu}_{s}$ and $\vec{\mu}_s$ are equivalent; that is, $\vec{\nu}_{s} \ll \vec{\mu}_{s}$ and vice versa.
\end{lemma}

The convergence of $\{R_{s,N}\}_{N\in \N}$ was essentially proved in \cite[Lemma 3.4]{OTz2}.
The uniform bound \eqref{unifbd} is possible in view of the positivity of the term $\frac{1}{4}\int u^4 dx$ in \eqref{RsN}. Our proof of Lemma~\ref{LEM:nus} is by now a standard application of the Bou\'e-Dupuis variational formula \cite{BD, Ust, BG}. 
However, to our knowledge a proof for the whole range $s>0$ does not appear in the literature, so for the reader's convenience, we will present some of the details in Section~\ref{SEC:BD}.

 Note that 
\begin{align}
Z_{s,N}:=\int e^{-R_{s,N}(u)}d\vec{\mu}_{s}(u) = \int e^{-R_{s,N}(u_N)}d\mu_{s,N}(u_N). \label{eRsN}
\end{align}
In particular, by Jensen's inequality and Lemma~\ref{LEM:nus}, there is a $C>0$ such that
\begin{align}
\inf_{N\in \N} Z_{s,N} \geq C>0. \label{ZNlower}
\end{align}

\subsection{Pathwise properties}

The goal of this section is to establish some properties that sample paths of $\vec{\nu}_s$ and $\xi$ satisfy, building up to the uniform bounds on the truncated flows $\Phi^{N}_{t}$ in Lemma~\ref{LEM:Calcontrol} and Lemma~\ref{LEM:Qscontrol}. 
We begin with samples of $\vec{\mu}_{s}$, which enjoy the following regularity property.

\begin{lemma}\label{LEM:Zmus}
Let $\al<s$.
For $\u\sim \vec{\mu}_{s}$, we have $\u\in Z^{\al}\subset \CC^{\al}$ almost surely.
Moreover, 
\begin{align}
\E \bigg[ \sup_{N\in 2^{\N_0} \cup\{\infty\}}\| \Pi_{\leq N} \Phi_{t}^{\textup{lin}}(\u_0,\xi)\|_{C([0,T];\CC^{\al})}^{p} \bigg] \leq C(p,T). \label{uniflin}
\end{align}
\end{lemma}
\begin{proof}
Clearly, the membership in $Z^{\al}$ is stronger than that in $\CC^{\al}$, so we focus on the former.  
Using Sobolev embedding for some $p<\infty$ large enough so that $\frac{2}{p}+\al<s$, with $q\geq p$, Lemma~\ref{LEM:Kolm}, we have
\begin{align*}
\E[ \|\u\|_{Z^{\al}}^{q}] & \les \sum_{j=0}^{\infty} e^{\frac{j}{8}} \E\Big[ \sup_{t\in [j,j+1]} \|S(t)\u\|_{\W^{\al+\frac{2}{p},p}}^{q} \Big] \\
& \les  \sum_{j=0}^{\infty} e^{\frac{j}{8}} \bigg\{ \E\Big[ \|S(j)\u\|_{\W^{\al+\frac{2}{p},p}}^{q}\Big] + \sup_{j\leq t' <t\leq j+1} \E\bigg[\frac{ \| S(t)\u - S(t')\u\|_{\W^{\al+\frac{2}{p},p}}^{q}}{|t-t'|^{1+\g q +\eps}} \bigg] \bigg\},
\end{align*}  
for some $\g>0$ to chosen later.
It will suffice to control the norms of just the first components of the vectors.  For each fixed $x\in \T^2$, $\jb{\nb}^{\al+\frac 2p} \pi_1 S(j)\u (x)$ is a mean zero Gaussian random variable with a variance bounded by 
\begin{align}
\text{Var}(\jb{\nb}^{\al+\frac 2p} \pi_1 S(j)\u (x))\les e^{-j} \sum_{n} \frac{\jb{n}^{2\al+\frac{4}{p}}}{\jb{n}^{2+2s}} \les e^{-j}. \label{varSu}
\end{align}
By Minkowski's inequality, the Gaussianity of $\jb{\nb}^{\al+\frac 2p} \pi_1 S(j)\u (x)$ and \eqref{varSu},
\begin{align*}
\E[ \|\pi_1 S(j)\u\|_{W^{\al+\frac{2}{p},p}}^{q}]^{\frac 1q} \leq  \bigg\|  \E[ |\jb{\nb}^{\al+\frac 2p} \pi_1 S(j)\u (x)|^{q}]^{\frac 1q}   \bigg\|_{L^{p}_{x}} \les e^{-\frac{j}{2}}q^{\frac 12}.
\end{align*}
The factor $e^{-j/2}$ ensures the summability of the first term in $j$. 
For the second term we argue similarly using the mean value theorem to see that for each fixed $x\in \T^2$, $\jb{\nb}^{\al+\frac 2p} \pi_1 S(t)\u (x)-\jb{\nb}^{\al+\frac 2p} \pi_1 S(t')\u (x)$ is a mean zero Gaussian random variable with variance bounded by 
\begin{align*}
e^{-j} \sum_{n\in \Z} \frac{\jb{n}^{2\ta+2\al+\frac{4}{p}}|t-t'|^{2\ta}}{ \jb{n}^{2+2s}  } \les e^{-j} |t-t'|^{2\ta},
\end{align*}
provided that $0<\ta< \min(1, s-\al-\frac{2}{p})$. Thus, 
\begin{align*}
 \sup_{j\leq t' <t\leq j+1} \E\bigg[\frac{ \| S(t)\u - S(t')\u\|_{\W^{\al+\frac{2}{p},p}}^{q}}{|t-t'|^{1+\g q +\eps}} \bigg]  \les e^{-q\frac{j}{2}} q^{\frac{q}{2}},
\end{align*}
provided that $\g<\ta-\frac{1}{q}-\frac{\eps}{q}$ which is possible by taking $q$ sufficiently large and $\eps$ sufficiently small at the beginning. 
We then obtain 
\begin{align*}
\E[ \|\u\|_{Z^{\al}}^{q}] \les q^{\frac 12}.
\end{align*}

As for \eqref{uniflin}, this follows in a similar way as for \eqref{momentsstick}. In particular, by \eqref{Philin} and the triangle inequality, we only need to prove the bound for $S(t)\u_0$, which further reduces to showing
\begin{align}
\E\big[ \| \Pi_{\leq 1}S(t)\u_0\|_{C([0,T];\CC^{\al})}^{p}\big] & \les_{p,T} 1,\label{diffu02}\\
\E \big[  \| \Pi_{\leq 2M}S(t)\u_0-\Pi_{\leq M}S(t)\u_0\|_{C([0,T];\CC^{\al})}^{p}] &\les_{p,T} M^{-\ta} \label{diffu01}
\end{align}
for some $\ta>0$, and every $M\in 2^{\N_0}$. Here, \eqref{diffu02} just follows from Sobolev embedding and \eqref{SonH}.
Now, \eqref{diffu01} follows similar arguments as above in showing $\u_0\in Z^{\al}$, and we gain the factor $M^{-\ta}$ from localisation to frequencies around $M$ in the sum \eqref{varSu}. 
\end{proof}


As our arguments in Section~\ref{SEC:lifts} require us to also control norms of the time evolution of the function $Q_{s,N}(u)$, we need to be more precise here about this quantity. For $u\in C^{\al}(\T^2)$, we define 
\begin{align}
Q_{s}(u) := \lim_{N\to \infty} Q_{s,N}(u) \label{Qslim}
\end{align}
in $H^{\al-s}(\T^2)$, whenever this limit exists. 

\begin{lemma} \label{LEM:Qprops}
Let $\al<s$ but sufficiently close to $s$. Suppose that
$u\in C^{\al}(\T^2)$ is such that $Q_{s}(u)$ exists in the sense of \eqref{Qslim}, and let $w\in H^{1+\al}(\T^2)$. Then, $Q_{s}(u+w)$ exists in the sense of \eqref{Qslim} and satisfies
\begin{align}
Q_{s}(u+w) = Q_{s}(u)  + 2(\jb{\nb}^s u)(\jb{\nb}^s w)+(\jb{\nb}^s w)^2. \label{Qsum}
\end{align}
In particular, if $u$ is distributed according to $\mu_{s+1}$, then almost surely $Q_{s}(u)$ exists, and moreover, $Q_{s}(u)\in C^{-\eps}(\T^2)$ for any $\eps>s-\al$, and $Q_{s,N}(u)\to Q_{s}(u)$ in $C^{-\eps}(\T^2)$. 
Finally,  for $\u_0$ distributed according to $\vec{\mu}_{s}$, $Q_{s}(\pi_1 \Phi_{t}^{\textup{lin}}(\u_0,\xi))$ exists almost surely 
and satisfies
\begin{align}
 \E\bigg[  \sup_{N\in 2^{\N_0}\cup \{\infty\}} \|Q_{s,N}(\pi_{1}\Phi_{t}^{\textup{lin}}(\u_0,\xi))\|_{C([0,T];H^{-\s})}^p\bigg] \leq C(T,p). \label{Qslin2}
\end{align}
for any $1\leq p<\infty$, $\s>0$, and where $\Phi_{t}^{\textup{lin}}$ is the linear flow  defined in \eqref{uN}.  
\end{lemma}

\begin{proof}
For each fixed $N\in \N$, it follows by the definition \eqref{Qs} that
\begin{align}
Q_{s,N}(u+w) = Q_{s,N}(u) + 2(\jb{\nb}^{s}\Pi_{\leq N}u)(\jb{\nb}^s \Pi_{\leq N}w)+ (\jb{\nb}^s \Pi_{\leq N}w)^2. \label{QsNsum}
\end{align}
Then, the existence of $Q_{s}(u+w)$ is ensured by verifying that the right hand side of \eqref{QsNsum} converges as $N\to \infty$. The term $Q_{s,N}(u)$ converges by our assumption on $u$. For the mixed term $(\jb{\nb}^{s}\Pi_{\leq N}u)(\jb{\nb}^s \Pi_{\leq N}w)$, we use \eqref{prod2} which is controlled provided that $1+\al-s>s-\al$, which holds as long as $\al<s$ is sufficiently close to $s$. Similarly, $(\jb{\nb}^{s} \Pi_{\le N}w)^{2}$ converges to $(\jb{\nb}^{s} w)^{2}$ in $H^{\al-s}$, as long as $1+\al-s>s-\al$.

The second claim regarding functions $u$ distributed according to $\mu_{s+1}$ is essentially shown in  \cite[(4.6), Proposition 4.3]{GOTW} and the required details are already encapsulated in showing \eqref{Qslin2}, which we move onto now.

 By a telescoping series argument, it is enough to show that
\begin{align}
\E\big[ \| Q_{s,1}(\pi_1 \Phi^{\text{lin}}_{t}(\u_0,\xi))\|_{C([0,T];H^{-\s})}^{p}\big] & \les_{p,T} 1,\label{Q0}\\
\E \big[  \| Q_{s,2M}(\pi_1 \Phi^{\text{lin}}_{t}(\u_0,\xi))-Q_{s,M}(\pi_1 \Phi^{\text{lin}}_{t}(\u_0,\xi))\|_{C([0,T];H^{-\s})}^{p}] &\les_{p,T} M^{-\dl} \label{Q1}
\end{align}
for $1\leq p<\infty$ and some $\dl>0$. In particular, \eqref{Q1} also establishes that the sequence $\{Q_{s}(\pi_1\Phi_{t}^{\text{lin}}(\u_0,\xi))\}_{N}$ is Cauchy in this sense and thus  $Q_{s}(\pi_{1} \Phi_{t}^{\text{lin}}(\u_0,\xi))$ exists in $L^p( \O; C([0,T];H^{-\s}))$. We only consider \eqref{Q1} as simpler arguments suffice for \eqref{Q0}.

By Lemma~\ref{LEM:Kolm}, we have 
\begin{align}
\begin{split}
&\E \big[  \| Q_{s,2M}(\pi_1 \Phi^{\text{lin}}_{t}(\u_0,\xi))-Q_{s,M}(\pi_1 \Phi^{\text{lin}}_{t}(\u_0,\xi))\|_{C([j,j+1];H^{-\s})}^{p}]  \\
& \les  \E \big[  \| \mathcal{Q}_{s,M}(t)\|_{H^{-\s}}^{p}]  +\sup_{j\leq t'<t\leq j+1}    \E\bigg[  \frac{ \| \mathcal{Q}_{s,M}(t)-\mathcal{Q}_{s,M}(t') \|_{H^{-\s}}^{p}   }{|t-t'|^{1+p\g +\eps}}   \bigg]    ,
\end{split} \label{Qsnorm1}
\end{align}
where $\g<\ta -\frac{1+\eps}{p}$, for some $0<\ta<1$,
 and $$\mathcal{Q}_{s,M}(t): =Q_{s,2M}(\pi_1 \Phi^{\text{lin}}_{t}(\u_0,\xi))-Q_{s,M}(\pi_1 \Phi^{\text{lin}}_{t}(\u_0,\xi)).$$
 For each fixed $t\geq 0$, the law of $\Phi_{t}^{\text{lin}}(\u_0,\xi)$ is $\vec{\mu}_{s}$, so then
\begin{align*}
\E\Big[ \|\mathcal{Q}_{s,M}(j)\|_{H^{-\s}}^{2} \Big] &= \E\Big[ \|\mathcal{Q}_{s,M}(0)\|_{H^{-\s}}^{2} \Big]\\
  &\les \sum_{n\neq 0} \frac{1}{\jb{n}^{2\s}}  \sum_{ \substack{n=n_{1}+n_2\\ \exists j\in \{1,2\}, |n_j|\sim M}} \prod_{j=1}^{2} \frac{1  }{\jb{n_1}^2 \jb{n_2}^{2}} + \sum_{M<|n|\leq 2M} \frac{1}{\jb{n}^{4}} \\
  & \les (\log M )M^{-2\s} +M^{-2},
\end{align*}
which is acceptable as $\s>0$. The Wiener Chaos estimate (see \cite[Theorem I.22]{Simon} and \cite[Proposition 2.4]{ThTz}) then implies a similar result for any finite $p>2$.
Moreover, by the Wiener Chaos estimate, for the second term in \eqref{Qsnorm1}, it is enough to show that 
\begin{align}
\E\Big[ \|\mathcal{Q}_{s,M}(t)-\Q_{s,M}(t')\|_{H^{-\s}}^{2} \Big] \les  |t-t'|^{2\ta}, \label{Qsdifft}
\end{align}
where $\g <\ta-\frac{1+\eps}{p}$ and for some $0<\ta<1$. By expanding the definition of the $H^{-\s}$ norm, \eqref{Qsdifft} follows from 
\begin{align}
\E\Big[ \big| \mathcal{F}\{\mathcal{Q}_{s,M}(t)-\mathcal{Q}_{s,M}(t')\}(n) \big|^{2} \Big] \les  |t-t'|^{2\ta} \jb{n}^{-2+2\ta+2\dl}M^{-\dl}, \label{Qsdifft2}
\end{align}
for any $n\in \Z$.
We compute
\begin{align}
\begin{split}
\mathcal{F}\{\Q_{s,M}(t)\}(n) &= 
\sum_{ \substack{n=n_{1}+n_2\\ n\neq 0 \\ \exists j\in \{1,2\}, |n_j|\sim M}}\prod_{j=1}^{2} \frac{z(t,n_j)}{\jb{n_j}} + \ind_{\{n=0\}} \sum_{M<|m|\leq 2M} \frac{|z(t,m)|^2-1}{\jb{m}^2}
\end{split} \label{FTQ}
\end{align}
where
\begin{align*}
z(t,n) : =  e^{-\frac{t}{2}} \bigg(   (   \cos(t\sbb{n})+\tfrac{\sin(t\sbb{n})}{2\sbb{n}} )g_n 
&+ \tfrac{\jb{n}}{ \sbb{n}}\sin(t\sbb{n})h_n  \\
&+ \tfrac{\sqrt{2}\jb{n}}{\sbb{n}} \int_{0}^{t} e^{\frac{t'}{2}}\sin((t-t')\sbb{n})d\be_{n}(t')
 \bigg).
\end{align*}
Then, taking the difference of the first term in \eqref{FTQ} at times $j\leq t'<t\leq j+1$, with the absolute value squared, and using the mean value theorem, we find
\begin{align}
\begin{split}
 \sum_{\substack{n=n_1+n_2 \\  \max(|n_1|,|n_2|)\sim M } } \frac{ \max(\jb{n_1},\jb{n_2})^{2\ta}}{\jb{n_1}^2\jb{n_2}^2} |t-t'|^{2\ta} 
 & \les |t-t'|^{2\ta} \jb{n}^{-2+2\ta+2\dl}M^{-\dl},
 \end{split} \label{sums1}
\end{align}
provided that $0<\ta<1$.
When $n=0$, we need to use that fact that for each fixed $t\geq 0$ and $m\in \Z$, $z(t,m)$ is a standard complex Gaussian random variable so 
\begin{align*}
|z(t,m)|^2 -1 = |z(t,m)|^2- \E[|z(t,m)|^2].
\end{align*}
Then, \eqref{Qsdifft2} follows in this case since 
\begin{align*}
\E[ (|z(t,m_1)|^2-|z(t',m_1)|^2)&(|z(t,m_2)|^2-|z(t',m_2)|^2)]  \\
& = \dl_{m_1,m_2} \E[ (|z(t,m_1)|^2-|z(t',m_1)|^2)^2],
\end{align*}
and thus 
\begin{align*}
\E\Big[ \big| \mathcal{F}\{\mathcal{Q}_{s,M}(t)-\mathcal{Q}_{s,M}(t')\}(0) \big|^{2} \Big] \les |t-t'|^{2\ta} \sum_{|m|\sim M} \frac{1}{\jb{m}^{4-2\ta}} \les |t-t'|^{2\ta} M^{-2+2\ta}
\end{align*}
for any $0<\ta<1$. This completes the proof of \eqref{Qsdifft2} and thus also \eqref{Qsdifft}.
\end{proof}

We now show that \eqref{uniflin} and \eqref{Qslin2} also hold for the nonlinear flow $\Phi_{t}$ of \eqref{SDNLW}. 
To this end, we need to exploit that $\w$ belongs to the higher regularity space $\H^{1+s-\eps}$ for any $\eps>0$, namely, one full degree better than that of the initial data $\u_0$ and the stochastic convolution $\stick_{t}(\xi)$. It is convenient to verify this a posteriori relative to the results from Proposition~\ref{PROP:LWP}. In particular, the following results hold pathwise.

\begin{lemma}\label{LEM:Calcontrol}
Let $0<\al'<\al<s$, $T>0$, and $\u_0 \in X^{\al}$. Then, for every $N\in 2^{\N_0}\cup\{\infty\}$, $\w_{N}(t)\in C([0,T];\H^{1+\al}(\T^2))$, and whenever 
\begin{align}
\sup_{N\in 2^{\N_0}\cup\{\infty\}} \sup_{t\in [0,T]} \Big( \| \Pi_{\leq N}\Phi_{t}^{\textup{lin}} (\u_0,\xi)\|_{\W^{\al,\frac{2}{\al}}} 
+ \|\Pi_{\leq N}\stick_{t}(\xi) \|_{Z^{\al}}  \Big)\leq R \label{wH1aAssumption}
\end{align}
for some $R>0$, there exists $C(T,R)>0$ such that 
\begin{align}
\sup_{N,M\in 2^{\N_0}\cup\{\infty\}}\sup_{t\in [0,T]} \Big\{ \| \Pi_{\leq M} \w_{N}(t)\|_{\CC^{\al'}}+   \| \w_{N}(t)\|_{\H^{1+\al}} \Big\} \leq C(T,R). \label{wH1a}
\end{align}
\end{lemma}

\begin{proof}
By \eqref{veqnNd}, \eqref{SonH}, the fractional Leibniz rule, and the definitions of the space $X^{\al}$, we get
\begin{equation}
\begin{aligned}
\|\w_N (t)\|_{\H^{1+\al}} &\les \int_{0}^{t}e^{-\frac{t-t'}{2}}\big(\|S(t') \u_0 + \stick_{t}(\xi)\|_{\X}^3 +\|\w_N(t')\|_{\H^{1}}^{3} \big) dt'. 
\end{aligned}
\label{tightv}
\end{equation}
Now, the bound in $\H^{1+\al}$ in \eqref{wH1a} follows from \eqref{wH1aAssumption} and \eqref{wNbd}. For the $\CC^{\al'}$-norm in \eqref{wH1a}, we use the embedding $\H^{1+\al'+\eps}\embeds \CC^{\al'}$ for any $\eps>0$, the boundedness of $\Pi_{\leq M}$ on $\H^{1+\al}$, and the bound in $\H^{1+\al}$ that we just proved.
\end{proof}

We now obtain a uniform control on the transport of the renormalised quantities $Q_{s,M}$.

\begin{lemma}\label{LEM:Qscontrol}
Let $0<\al<s$ such that $s-\al < \frac 12$. Given $T>0$ and $R> 0$, suppose that
\begin{align*}
\sup_{M\in 2^{\N_0}}\sup_{t\in [0,T]}\big( \| Q_{s,M}(\pi_1 \Phi_{t}^{\textup{lin}}(\u_0,\xi))\|_{H^{\al-s}} + \|\Pi_{\leq M}\Phi_{t}^{\textup{lin}}(\u_0,\xi)\|_{\CC^{\al}} + \| \Pi_{\leq M} \stick_{t}(\xi)\|_{ Z^{\al}} \big)\leq R.
\end{align*}
Then, there exists a constant $C(T,R)>0$ such that 
\begin{align}
\sup_{N,M\in 2^{\N_0}}\sup_{t\in [0,T]} \| Q_{s,M}(\pi_1 \Phi_{t}^{N}(\u_0,\xi))\|_{H^{\al-s}} \leq C(T,R). \label{Qsnonlinbd}
\end{align}
\end{lemma}
\begin{proof}
By Sobolev embedding, our assumption, and \eqref{wH1a},
\begin{align*}
\| (\jb{\nb}^s \pi_1 \Pi_{\leq M} \w_{N}(t))^{2}\|_{H^{\al-s}} &\leq \| \pi_1 \Pi_{\leq M}\w_{N}(t)\|_{W^{s,4}}^{2} \les \|\pi_1 \w_{N}(t)\|_{H^{s+\frac 12}}^{2}    \leq C(R,T),
\end{align*}
uniformly in $M$ and $N$, and where we used that $s+\frac{1}{2}\leq 1+\al$, when $s-\al \leq \frac 12$.
Next, using \eqref{prod2} and \eqref{wH1a}, we get 
\begin{align*}
\| (\jb{\nb}^s \pi_1 &\Pi_{\leq M}\Phi_{t}^{\text{lin}}(\u_0,\xi))( \jb{\nb}^{s}\pi_{1} \Pi_{\leq M}\w_N (t)) \|_{H^{\al-s}}  \\
&\les \|\jb{\nb}^s \pi_1 \Pi_{\leq M}\Phi_{t}^{\text{lin}}(\u_0,\xi)\|_{C^{\al-s}} \| \jb{\nb}^{s}\pi_{1} \w_N (t)\|_{H^{\al-s+\eps}} \\
&\les \|\pi_1 \Pi_{\leq M}\Phi_{t}^{\text{lin}}(\u_0,\xi)\|_{C^{\al}} \| \pi_{1} \w_{N}(t)\|_{H^{1+\al}} \\
& \leq C(R,T),
\end{align*}
provided that $\eps<1$.
Combining these estimates, recalling \eqref{PhiN} and \eqref{uN} and then applying \eqref{Qsum} and the triangle inequality yields \eqref{Qsnonlinbd}.
\end{proof}

\subsection{Proof of Lemma~\ref{LEM:nus}} \label{SEC:BD}

First, we note that \eqref{eRLpconv} follows from \eqref{unifbd} and the convergence in measure from of \eqref{RsN}, which in turn follows from immediately from Lemma~\ref{LEM:Qprops}. We thus focus on establishing \eqref{unifbd}. We set $p=1$ as the argument is identical for any finite $p>1$.
We now state the Bou\'e-Dupuis variational formula \cite{BD, Ust}; the current version can be found in    \cite[Lemma 2.6]{FT2}.

\begin{lemma}\label{LEM:var3}
Let $Y$ be distributed according to $(\pi_{1})_{\ast}\vec{\mu}_{s}$ and $N \in \N$.
Suppose that  $F:C^\infty(\T^2) \to \R$
is measurable such that $\E\big[|F_{-}(\Pi_{\leq N}Y)|^p\big] < \infty$ for some $p>1$, where $F_{-}$ denotes the negative part of $F$. Then, we have
\begin{align}
\log \E\Big[e^{-F(\Pi_{\leq N} Y)}\Big]
\leq 
\E\bigg[ \sup_{\Dr\in H^{1+s}}\Big\{ -F(\Pi_{\leq N} Y + \Pi_{\leq N} \Dr) - \tfrac{1}{2} \| \Dr \|_{H^{1+s}_x}^2 \Big\} \bigg], 
\label{P3}
\end{align}
\noi
and the expectation $\E = \E_{\mathbb{P}}$
is an 
expectation with respect to the underlying probability measure~$\mathbb{P}$. 
\end{lemma}

In the following, we simply write $Y_{N}$ for $\Pi_{N}Y$ and $\Dr_{N}$ for $\Pi_{\leq N}\Dr$.

\begin{lemma}  \label{LEM:Dr}
	Let $s>0$.
Then, given any finite $p \ge 1$, 
we have 
\begin{align*}
\begin{split}
\E 
\Big[ & \|Y_N\|_{C^{s-\eps}}^p
+ \| Q_{s,N}(Y_N)\|_{C^{-\eps}}^p
\Big]
\leq C_{\eps, p} <\infty,
\end{split}
\end{align*}
\noi
 for any $\eps>0$, and uniformly in $N \in \N$ 
\end{lemma}
We omit the proof of Lemma~\ref{LEM:Dr} as it follows similarly from \cite[(4.6), Proposition 4.3]{GOTW}.

\begin{proof}[Proof of Lemma~\ref{LEM:nus} \eqref{unifbd}]

In view of \eqref{eRsN}, we apply Lemma~\ref{LEM:var3} with $F=R_{s,N}$. Note that the assumptions in Lemma~\ref{LEM:var3} are trivially satisfied as the bounds may depend on $N$. 
From \eqref{P3}, we have 
\begin{align}
\begin{split}
\log Z_{s,N} 
& \leq\E \Big[ \sup_{\Dr\in H^{1+s}}\Big\{ -R_{s,N}(Y_N +\Dr_N) -\tfrac{1}{2}\|\Dr_N\|_{H^{1+s}}^{2} \Big\}\Big].
\end{split} \label{varform1}
\end{align}
By H\"{o}lder and Young's inequalities, we have 
\begin{align}
\bigg| \int_{\T^2} Y_{N}^{3} \Dr_{N} dx \bigg| +\bigg| \int_{\T^2} Y_{N} \Dr_{N}^3 dx \bigg|& \leq C\|Y_N\|_{L^{4}}^{4} +\frac{1}{100}\|\Dr_{N}\|_{L^{4}}^{4}. \label{L4part}
\end{align}
Thus, \eqref{L4part} implies 
\begin{align}
-\frac{1}{8}\int_{\T^2}(Y_N+\Dr_N)^{4}dx \leq -\frac{1}{16}\|\Dr_N\|_{L^{4}}^{4} + C\|Y_N\|_{L^{4}}^{4}. \label{L4part2}
\end{align}
Next, by \eqref{duality}, \eqref{Leibniz}, \eqref{Besovembeds}, for every $\dl>0$, there exists $C_{\dl}>0$ such that
\begin{align}
\begin{split}
\bigg| \int_{\T^2} Q_{s,N}(Y_N)(Y_N +\Dr_N)^2 dx \bigg| & \leq  \| Q_{s,N}(Y_N)\|_{C^{-\frac{\eps}{2}}} \| (Y_N+\Dr_N)^{2}\|_{B^{\frac{\eps}{2}}_{1,1}} \\
& \les \| Q_{s,N}(Y_N)\|_{C^{-\frac{\eps}{2}}}\|Y_N+\Dr_N\|_{L^2} \|Y_N+\Dr_N\|_{B^{\frac{\eps}{2}}_{2,1}} \\
& \les \| Q_{s,N}(Y_N)\|_{C^{-\frac{\eps}{2}}}\|Y_N+\Dr_N\|_{L^2} \|Y_N +\Dr_N\|_{H^{\eps}} \\
& \leq C_{\dl} (1+\| Q_{s,N}(Y_N)\|_{C^{-\frac{\eps}{2}}})^{2}(1+\|Y_N\|_{H^{\eps}})^{4}  \\
&\hphantom{XXX}+ \dl( \|\Dr_N\|_{L^{4}}^{4} +\|\Dr_N\|_{H^{1+s}}^{2}).
\end{split}
\label{Dsterm1}
\end{align}
By H\"{o}lder and Young's inequalities, we also have
\begin{align}
\bigg| \int_{\T^2} Y_N \Dr_N (Y_N +\Dr_N)^2 dx \bigg| & \leq C_{\dl}(1+\|Y_N\|_{L^{\infty}})^{4} + \dl\|\Dr_N\|_{L^{4}}^{4},
\label{Dsterm2}
\end{align}
for any $\dl>0$.
Thus, by \eqref{Dsterm1} and \eqref{Dsterm2}, we have
\begin{align}
\begin{split}
&-\int_{\T^2} Q_{s,N}(Y_N+\Dr_N)(Y_N +\Dr_N)^2 dx \\
& \leq -  \int_{\T^2} Q_{s,N}(Y_N)(Y_N +\Dr_N)^2 dx -2\int_{\T^2} Y_N \Dr_N (Y_N +\Dr_N)^2 dx \\
& \leq C( \|Q_{s,N}(Y_N)\|_{C^{-\frac{\eps}{2}}}, \|Y_N\|_{H^{\eps}}, \|Y_{N}\|_{L^{\infty}}) + \dl ( \|\Dr_N\|_{L^{4}} + \|\Dr_{N}\|_{H^{1+s}}^{2}),
\end{split} \label{Dsterm3}
\end{align} 
where $C(x,y,z)$ is polynomial in its arguments. Therefore Lemma~\ref{LEM:Dr} implies 
\begin{align*}
\sup_{N\in \mathbb{N}}\E[ C( \|Q_{s,N}(Y_N)\|_{C^{-\frac{\eps}{2}}}, \|Y_N\|_{H^{\eps}}, \|Y_{N}\|_{L^{\infty}}) ] <\infty.
\end{align*}
Combining \eqref{L4part2} and \eqref{Dsterm3} in \eqref{varform1} and choosing $\dl>0$ sufficiently small, we have shown 
\begin{align*}
\log Z_{s,N} \leq C+ \E\big[ \sup_{\Dr \in H^{1+s}}\big\{ -\tfrac{1}{32}\|\Dr_N\|_{L^{4}}^{4} -\tfrac{1}{4}\|\Dr_N\|_{H^{1+s}}^{2} \big\} \big] \leq C,
\end{align*}
uniformly in $N\in \mathbb{N}$. This completes the proof of \eqref{unifbd}.
\end{proof}

\section{On the transported densities}

\subsection{Functional derivatives}

For vector-valued functions $\u $ and $\v$, we define an inner product on $L^2(\T^2)\times L^2(\T^2)$  by
\begin{align*}
\jb{ \u, \v}_{L^2\times L^2} : = \jb{u_1, v_1}_{L^2} + \jb{u_2,v_2}_{L^2},
\end{align*}
and we denote the Fr\'{e}chet derivatives as follows: if $F:\S(\R)^2 \to \R$, then
\begin{align*}
dF\vert_{\u}[ {\bf f}] = \frac{d}{d\ta}\bigg\vert_{\ta=0} F(\u+\ta {\bf f}) =\jb{ D_{\u}F, {\bf f}} = \jb{\pa_{u_1}F, f_1}+\jb{\pa_{u_2}F, f_2}.
\end{align*}
and, if $F$ is also $C^2$, then we define the second functional derivative
\begin{align*}
D^{2}F\vert_{\u }[{\bf f}, {\bf h}] := \frac{d}{d\ta}\bigg\vert_{\ta=0} dF|_{\u+\ta {\bf h}} [ {\bf f}] = \jb{ (D^{2}_{\u} F)[ {\bf h}], {\bf f}}
\end{align*}
and the Hessian $D^{2}_{\u}F$ can be further expanded as
\begin{align*}
D^{2}_{\u}F = 
\begin{pmatrix} \pa_{u_1 u_2}^{2} F & \pa_{u_1 u_2}^{2}F \\  \pa_{u_2 u_1}^{2}F & \pa_{u_2 u_2}^{2}F \end{pmatrix}.
\end{align*}

In the finite dimensional setting, we can relate the functional derivatives $D_{\u_N}$ to the derivatives with respect to the Fourier coefficients $\{a_0, a_{n,1}, a_{n,2}\}_{n\in \Ld_{N}^{\ast}}$ of $u_N$ as in \eqref{uNdecomp}. In particular, by the chain rule, we have
\begin{align*}
\pa_{a_{n,1}}F = \Re \ft{\pa_{u_N}F}(n), \quad \pa_{a_{n,2}}F = \Im \ft{\pa_{u_N}F}(n), \quad \text{and} \quad \pa_{a_0}F =  \ft{\pa_{u_N}F}(0),
\end{align*}
so that 
\begin{align*}
\pa_{u_N}F(x) = \pa_{a_{0}}F + \sum_{n\in \Ld_{N}^{\ast}} \big\{ \pa_{a_{n,1}}F \cdot 2\cos(n\cdot x) + \pa_{a_{n,2}}F \cdot (-2\sin(n\cdot x))\big\}.
\end{align*}

\subsection{Transported densities}
In the following, we fix $N\in 2^{\N_0}$ and consider the truncated system  \eqref{NLWuN}. 
Using $\Pi_{\leq N} \xi (t) = \sum_{|n|\leq N} d\be_{n}(t)e^{in\cdot x}$,
\eqref{NLWuN} is then a finite dimensional system of SDEs for the Fourier coefficients $\{ (a_0, a_{n,1},a_{n,2}) ,(b_0, b_{n,1},b_{n,2})\}_{n\in \Ld_{N}^{\ast}}$. 
We use \eqref{uNdecomp} as a bijection between $(u_N,v_N)$ and the coefficients $\{ (a_0, a_{n,1},a_{n,2}) ,(b_0, b_{n,1},b_{n,2})\}_{n\in \Ld_{N}^{\ast}}$. 
We write \eqref{NLWuN} in the Ito formulation:
\begin{align}
d\vect{u_N}{v_N} = \bigg\{   J\nb_{(u_N, v_N)}H(u_N,v_N) -   \vect{0}{v_N}  \bigg\}dt +\sqrt{2} \begin{pmatrix} 0 & 0 \\  0 & \Pi_{\leq N}\jb{\nb}^{-s} \end{pmatrix}  \vect{0}{dW_{t}}, \label{NLWsde}
\end{align}
where 
\begin{align*}
J : = \begin{pmatrix} 0 & 1 \\  -1 & 0 \end{pmatrix}, 
\end{align*}
and $\nb_{(u_N,v_N)} = (\pa_{u_N}, \pa_{v_N})$ is the gradient with respect to the variables $(u_N,v_N)$. 
With respect to $\mathcal{V}_{N} \times \mathcal{V}_N$, we associate to \eqref{NLWsde} 
the (adjoint of the) infinitesimal generator:
\begin{align}
\begin{split}
\L_N^{\ast} F &= -\text{div}_{(u_N,v_N)} \big[ (J\nb H)F \big]   +\text{div}_{(u_N,v_N)} \bigg[ \begin{pmatrix} 0 & 0 \\  0 & 1 \end{pmatrix} \vect{u_N}{v_N} F  \bigg] \\
 &\hphantom{XXX}+ \text{Tr}\bigg[   D_{(u_N,v_N)}^{2} \begin{pmatrix} 0 & 0 \\  
 0 & \Pi_{\leq N}\jb{\nb}^{-2s} \end{pmatrix}F    \bigg]. 
 \end{split} \label{gen}
\end{align}
As $\text{div}_{(u_N,v_N)} J\nb H=0$, we rewrite \eqref{gen} as 
\begin{align}
\L_N^{\ast} F &= -\{H, F\}   + \text{div}_{v_N}( v_N F) +  \text{div}_{v_N}[ \nb_{v_N}( \jb{\nb}^{-2s}F)], \label{gen2}
\end{align}
where 
\begin{align}
\{H,F\}  :  = \jb{ \pa_{v_N}H, \pa_{u_N}F} - \jb{\pa_{u_N}H, \pa_{v_N}F}, \label{Poisson}
\end{align}
and 
\begin{align*}
\text{div}_{v_N} [ F(v)]: =  \text{div}_{(b_0, b_{1,n},b_{2,n})} F(v) = \pa_{b_0} F +
 \sum_{n=1}^{N} \big( \pa_{b_{1,n}}F + \pa_{b_{2,n}}F\big).
\end{align*}
Note that
\begin{align*}
\text{div}_{v_N} [ v_N F] &= (|\Ld_{N}^{\ast}|+1)F + \jb{v_N, \pa_{v_N}F}.
\end{align*}

Let $\u_{N}(t)$ denote the almost sure global-in-time solution to \eqref{NLWsde} with initial data $\u_{0,N}$, as guaranteed by Proposition~\ref{PROP:LWP}, and recall the decomposition
\begin{align*}
\u_N (t; \u_0,\xi)  =S(t) \u_{0,N} + \Pi_{\leq N} \stick_{t}(\xi) + \w_{N}(t).
\end{align*}
These solutions generate a Markov semigroup
\begin{align*}
\wt{\mathcal{P}}_{t}F(\u_{0,N}) = \E[ F( \u_N (t; \u_0,\xi) )],
\end{align*}
and the push-forward measure $\g_{t,N}:=(\wt{\mathcal{P}}_{t})_{\ast} L_N$ is a weak solution to the Fokker-Planck equation:
\begin{equation}
\begin{cases}\label{FP}
\dt \g_{t,N} =\L^{\ast} \g_{t,N} \\
\g_{t,N} \vert_{t=0} = e^{-E_{s,N}(u_N,v_N)} dL_{N},
\end{cases} 
\end{equation}
where $E_{s,N}$ is as in \eqref{EsN}. See \cite[Proposition 1.3.1]{BKRS}.



\begin{lemma}\label{LEM:htN}
For each $N\in 2^{\N_0}$, there exists $h_{t}^{N} \in C^{1,2}([0,\infty)\times \R^{2(|\Ld_{N}^{\ast}|+1)})$  such that $\g_{t,N} = h_{t}^{N} dL_{N}$, and $h_{t}^{N}$ satisfies
\begin{align}
\dt h_{t}^{N}(\u_{0,N}) \vert_{t=0} = \L^{\ast} ( e^{-E_{s,N}(\u_{0,N})})=-\{H, \EE_{s,N}\}(\u_{0,N})e^{-E_{s,N}(\u_{0,N})}, \label{ht0}
\end{align}
and
\begin{align}
 \{ H,\EE_{s,N}\}(\u_{0,N})=  \jb{v_N, 3\jb{\nb}^{s}[ \jb{\nb}^s u_N \cdot u_N^2]-\jb{\nb}^{2s}(u_N^3) + 3Q_{s,N}(u_N)u_N}.  \label{HEbracket}
\end{align}
\end{lemma}

\begin{proof}
As the coefficients of the operator $\L^{\ast}$ are smooth and the initial datum $e^{-E_{s,N}(\u_{0,N})}$ is smooth, H\"{o}rmander's condition ensures that $\g_{t,N}=h_{t}^{N}dL_{N}$ and
 $h_{t}^{N}\in C^{\infty, \infty}([0,\infty)\times \R^{2(|\Ld_{N}^{\ast}|+1)})$.
The first equality in \eqref{ht0} follows from \eqref{FP} and the second equality follows from \eqref{gen2}, \eqref{Poisson}, \eqref{EsN}, and the fact that $\{H,H\}=0$.  
 In particular, we have
\begin{align*}
 \text{div}_{v_N}[ \nb_{v_N}( \jb{\nb}^{-2s}e^{-E_{s,N}})]=-(|\Ld_{N}^{\ast}|+1)e^{-E_{s,N}} - \jb{v_N, \pa_{v_N}e^{-E_{s,N}}} =- \text{div}_{v_N}( v_N e^{-E_{s,N}}).
\end{align*}
Finally, for \eqref{HEbracket}, a direct computation from \eqref{EEsN} shows that 
\begin{align*}
\pa_{u_N}\EE_{s,N} (u_N,v_N) &=  (\jb{\nb}^{2s}-1)\jb{\nb}^{2}u_{N}+ 3 \jb{\nb}^{s}( \jb{\nb}^{s}u_N \cdot u_N^2) +3 Q_{s,N}(u_N) u_N, \\
 \pa_{v_N}\EE_{s,N} (u_N,v_N) &= \jb{\nb}^{2s}v_N-v_{N},
\end{align*}
and from \eqref{HN}, we have 
\begin{align*}
\pa_{u_N}H(u_N,v_N) = u_N - \Delta u_N +u_N^3 \quad \text{and} \quad \pa_{v_N}H(u_N,v_N)  = v_N.
\end{align*}
Then \eqref{HEbracket} follows from \eqref{Poisson}, and using that $\jb{\nb}^{2}=1-\Delta$. 
\end{proof}

We define
\begin{align}
g_{t}^{N} (\u_{0,N}) := e^{E_{s,N}(\u_{0,N})} h_{t}^{N} (\u_{0,N}) , \label{gtN}
\end{align}
which is the transported density of the truncated probability measure 
\begin{align}
d\vec{\rho}_{s,N}: = \wt{Z}_{s,N}^{-1} e^{-E_{s,N}(\u_{0,N})} dL_{N}= Z_{s,N}^{-1} e^{-R_{s,N}(\u_0)} d\vec{\mu}_{s,N}. \label{rhosN}
\end{align}
We extend $g_{t}^{N}$ to all of $X^{\al}$ by defining $g_{t}^{N}(\u_0) := g_{t}^{N}(\Pi_{\leq N}\u_0)$. 
In the following key result, we show that $g_{t}^{N}$ is the transported density of the measure $ (\Pt{t}^{N})^{\ast}\vec{\nu}_{s,N}$ with respect to $\vec{\nu}_{s,N}$ as well as obtain good long-time $L^p$-bounds for $g_{t}^{N}$ when restricted to compact sets.

\begin{lemma}\label{LEM:gtN}
Let $s>0$ and $N\in 2^{\N_0}$.
Then, it holds that 
\begin{align}
g_{t}^{N}(\u_0) = \frac{d  (\Pt{t}^{N})^{\ast}\vec{\nu}_{s,N}  }{ d\vec{\nu}_{s,N}}.
\label{gnusN}
\end{align}
Moreover,  for $\al>0$ such that $s-\al<\frac{1}{3}\min(1,s)$ and $K\subseteq \R^{2(\Ld_{N}^{\ast}+1)}$ a compact set, there exists $0<\tau(K,N)\leq 1$ such that
\begin{align}
|g_{t}^{N}(\u_{0,N})|\leq 1+ t ( \{H,\EE_{s,N}\}(\u_{0,N})+1) \label{gtNbd}
\end{align}
for all $0\leq t \leq \tau(K,N)$ and $\u_{0,N}\in K$.
Furthermore, given $R>0$, let
\begin{align}
K_{R}:= \Big\{ \u_{0}\,:\, \sup_{N\in \N}\big( \| \Pi_{\leq N} \u_{0}\|_{\CC^{\al}}+\| Q_{s,N}(\pi_1 \u_{0})\|_{H^{\al-s}} \big)\leq R\Big\}\label{KR}
\end{align}
and $K_{R,N}:=\Pi_{\leq N} K_{R}$.
Then, $K_{R,N}$ is pre-compact in $\R^{2(|\Ld_{N}^{\ast}|+1)}$ and 
\begin{align}
\sup_{0\leq t \leq \tau(K_{R,N},N)}\| g^{N}_{t}(\u_0)  \ind_{K_{R,N}}(\Pi_{\leq N}\u_0) \|_{L^{p}( \vec{\nu}_{s,N})}^{p}  \leq c_0 e^{ C(R)(p\tau(K_{R,N},N))^2}
\label{gbd}
\end{align}
for all $1\leq p <\infty$, and where the constants $c_0, C(R)$ do not depend on $N\in 2^{\N_0}$.
\end{lemma}

\begin{proof}
We first show \eqref{gnusN}. Recalling \eqref{nusN} and \eqref{mudecomp}, we have 
\begin{align}
d\vec{\nu}_{s,N}  =\vec{\rho}_{s,N}\otimes \vec{\mu}_{s,N}^{\perp}. \label{nusNdecomp}
\end{align}
Let $A \subseteq \Pi_{\leq N}X^{\al}$ and $B\subseteq \Pi_{>N}X^{\al}$ be measurable sets. 
Then, as $\Pi_{\leq N} \xi$ and $\Pi_{>N}\xi$ are independent and, for fixed $\u_0$, $\Pi_{\leq N}\Phi_{t}^{N}(\u_0,\xi)$ and $\Pi_{>N} \Phi^{\text{lin}}_{t}(\u_0,\xi)$ depend only on $\Pi_{\leq N} \xi$ and $\Pi_{>N}\xi$, respectively, the random variables $\Pi_{\leq N}\Phi_{t}^{N}(\u_0,\xi)$ and $\Pi_{>N} \Phi^{\text{lin}}_{t}(\u_0,\xi)$ are independent. Thus, using \eqref{nusNdecomp} and the invariance of $\vec{\mu}_{s,N}^{\perp}$ under the linear high-frequency flow $\Pi_{>N} \Phi^{\text{lin}}_{t}(\u_0,\xi)$, we obtain
\begin{align*}
\int &\ind_{A}(\Pi_{\leq N}\u_0) \ind_{B}(\Pi_{>N}\u_0) d(\Pt{t}^{N})^{\ast}(\vec{\nu}_{s,N})  \\
&   = \int \E\big[   \ind_{A}( \Pi_{\leq N}\Phi_{t}^{N}(\u_0,\xi)) \ind_{B}( \Pi_{>N} \Phi_{t}^{\text{lin}}(\u_0,\xi))\big]  d\vec{\nu}_{s,N} \\
& = \int \E[ \ind_{A}( \Pi_{\leq N}\Phi_{t}^{N}(\u_0,\xi))  ]  \E[\ind_{B}( \Pi_{>N} \Phi_{t}^{\text{lin}}(\u_0,\xi))] d\vec{\nu}_{s,N} \\
& = \int \E[ \ind_{A}( \Pi_{\leq N}\Phi_{t}^{N}(\u_0,\xi))  ] \int \E[\ind_{B}( \Pi_{>N} \Phi_{t}^{\text{lin}}(\u_0,\xi))] d\vec{\nu}_{s,N}^{\perp} d\vec{\rho}_{s,N} \\
&  = \int \ind_{B}(\Pi_{> N}\u_0) \int \E[ \ind_{A}( \Pi_{\leq N}\Phi_{t}^{N}(\u_0,\xi))  ]  d\vec{\rho}_{s,N} d\vec{\mu}_{s,N}^{\perp} \\
& = \int \ind_{A}(\Pi_{\leq N}\u_0) \ind_{B}(\Pi_{>N}\u_0) g_{t}^{N}(\u_0) d\vec{\nu}_{s,N}.
\end{align*}
As this holds for all such sets $A$ and $B$, we conclude \eqref{nusNdecomp}.

We now move onto the latter claims in the lemma.
By Lemma~\ref{LEM:htN} and \eqref{gtN}, $g_{t}^{N}\in C^{1,2}([0,\infty)\times \R^{2(|\Ld_{N}^{\ast}|+1)})$ and satisfies $g_{0}^{N}(\u_{0,N})=1$ and $\dt g_{t}^{N}(\u_{0,N})\vert_{t=0} =  -\{H,\EE_{s,N}\}(\u_{0,N})$. As $\dt g_{t}(\u_{0,N})$ is uniformly continuous on $[0,1]\times K$, there exists $0<\tau(K,N)\leq 1$ such that for all $0\leq t \leq \tau(K,N)$, and $\u_{0,N}\in K$,
\begin{align}
 |\dt g_{t}^{N}(\u_{0,N}) - \dt g_{t}^{N}(\u_{0,N})\vert_{t=0}| \leq 1. \label{gtNbd1}
\end{align}
For $\u_{0,N}\in K$, and $0\leq t\leq \tau(K,N)$, Taylor's theorem implies
\begin{align*}
g_{t}^{N}(\u_{0,N})= 1+ t \big[ (\dt g_{t}^{N}(\u_{0,N}))\vert_{t=0}+(\dt g_{t}^{N}(\u_{0,N}))\vert_{t=t'}-(\dt g_{t}^{N}(\u_{0,N}))\vert_{t=0}   \big],
\end{align*}
for some $0\leq t'\leq \tau(K,N)$. Then, \eqref{gtNbd} follows by using \eqref{gtNbd1}.

We move onto the claims regarding the set $K_{R,N}$. Since the $\CC^{\al}$ norm controls the Fourier coefficients of $\u_{0,N}$, $K_{R,N}$ is bounded, and hence pre-compact. We now show the bound \eqref{gbd}. Let $\tau_{R}:=\tau(K_{R,N},N)$.
From \eqref{gtNbd}, the inequality $1+x\leq e^{x}$ for any $x\in \R$, \eqref{unifbd}, \eqref{mudecomp} and Cauchy-Schwarz, we have
\begin{align}
\| g^{N}_{t}(\u_{0}) & \ind_{K_{R,N}}(\Pi_{\leq N}\u_{0}) \|_{L^{p}( \vec{\nu}_{s,N})}^{p}\notag \\
 & \leq \int |1+\tau_{R} ( \{H,\EE_{s,N}\}(\u_{0,N})+1)    |^{p} \ind_{K_{R,N}}(\u_{0,N}) d\vec{\nu}_{s,N} \notag \\
 & \leq \int e^{p \tau_{R}( \{H,\EE_{s,N}\}(\u_{0,N})+1)}\ind_{K_{R,N}}(\u_{0,N}) d\vec{\nu}_{s,N}  \notag \\
 & \leq Z_{s,N}^{-1} e^{p\tau_{R}} \| e^{-R_{s,N}(\u)}\|_{L^{2}(\vec{\mu}_{s})} \bigg( \int  e^{2 p\tau_{R} \{H,\EE_{s,N}\}(\u_{0,N})}\ind_{K_{R,N}}(\u_{0,N}) d\vec{\mu}_{s,N} \bigg)^{\frac 12}. \label{cs1}
\end{align}
Recall that for $\phi \in C^{\infty}(\T^2)$, the random variable $\jb{v_{N}, \phi}$ is a centered complex-valued Gaussian with variance $\| \Pi_{\leq N}\phi\|_{H^{-s}(\T^2)}^{2}$. 
Moreover from \eqref{HEbracket}, $\{H,\EE_{s,N}\}$ can be written as $\jb{\jb{\nb}^{s}v_N,  \mathcal{W}(u_N)}$, where
\begin{align*}
\mathcal{W}(u_N)= 3 [ \jb{\nb}^s u_N \cdot u_N^2]-\jb{\nb}^{s}(u_N^3) + 3 \jb{\nb}^{-s}(Q_{s,N}(u_N)u_N).
\end{align*}
Thus, from \eqref{HEbracket} and the independence of the random variables $v_{N}$ and $u_{N}$, we have
\begin{align*}
&\int  e^{2 p \tau_{R} \{H,\EE_{s,N}\}(\u_{0,N})}\ind_{K_{R,N}}(\u_{0,N}) d\vec{\mu}_{s,N}\\
 & \leq  \int \ind_{\pi_1 K_{R,N}}(u_{N})   \int  e^{2 p \tau_{R} \{H,\EE_{s,N}\}(\u_{0,N})}d{\mu}_{s,N} (v_N)   d\mu_{s+1,N}(u_N)  \\
 & = \int \ind_{\pi_1 K_{R,N}}(u_{N})e^{ 2(p\tau_{R})^2 \|\mathcal{W}(u_N)\|_{L^2}^{2}  } d\mu_{s+1,N}(u_N).
\end{align*}
It follows from \eqref{prod1}, \eqref{comL2}, and the definition of $K_{R,N}$, that there exists $C>0$ such that 
 \begin{align*}
 \|\mathcal{W}(u_N)\|_{L^2}^{2} \leq C(R^6+R^4).
\end{align*}
for $u_{N}\in  \pi_1 K_{R,N}$. Inserting this bound back into \eqref{cs1} and using \eqref{ZNlower} yields 
\eqref{gbd}.
\end{proof}

\subsection{Lifted Markov semigroups}\label{SEC:lifts}
For $\eps>0$ and $0<\al<s$, we define the space
\begin{align*}
\mathcal{Z} = \{ \xi \in C^{-1/2-\eps}_{t} C^{-1-\eps}_{x}([0,\infty)\times \T^2) \, :\,  \stick_{t}(\xi) \in C([0,\infty); Z^{\frac{\al+s}{2}})\},
\end{align*}
with the norm 
\begin{align*}
\| \xi\|_{\mathcal{Z}}:=  \|\xi\|_{C^{-1/2-\eps}([0,\infty); C^{-1-\eps})} + \|\stick_{t}(\xi)\|_{C([0,\infty);Z^{\frac{\al+s}{2}})}.
\end{align*}
It is clear from \eqref{sticktested} that $(\mathcal{Z}, \|\cdot\|_{\mathcal{Z}})$ is a Polish space.
 We view 
 $\Xi : = \text{Law}( \ind_{[0,\infty)\times \T^2}\xi)$ as a probability measure on $\mathcal{Z}$. More precisely, for any $F:\mathcal{Z}\to \R$, we have
\begin{align*}
\int F(u)d\Xi(u) =\E[ F(\xi)].
\end{align*}
In the following proposition, we show that for each $N\in  2^{\N_0}\cup\{\infty\}$, the Markov semigroups $\{\Pt{t}^{N}\}_{t\geq 0}$ lift to Markov semigroups on the product space $(\X \times \mathcal{Z}, \vec{\nu}_{s,N}\otimes \Xi)$. 

\begin{proposition}\label{PROP:lift}  
Given $N\in 2^{\N_0}\cup\{\infty\}$, there exists a Markov semigroup 
$\mathcal{U}^{N}_{t}$ on $\X \times \mathcal{Z}$ such that 
\begin{align}
(\Pt{t}^{N})^{\ast}\vec{\nu}_{s,N} \otimes \Xi  = (\mathcal{U}_{t}^{N})^{\ast}( \nu_{s,N}\otimes \Xi). \label{liftedmeas}
\end{align}
Moreover, the Radon-Nikodym derivative 
\begin{align}
f_{t}^{N}(\u_0, \xi)  : = \frac{d (\U_{t}^{N})^{\ast}( \vec{\nu}_{s,N}\otimes \Xi)}{d (\vec{\nu}_{s,N}\otimes \Xi)}. \label{ftNdef}
\end{align}
exists and is equal to $g_{t}^{N}(\u_0)$ for almost every $(\u_{0},\xi)$ with respect to $\vec{\nu}_{s,N}\otimes \Xi$, where $g_{t}^{N}$ was defined in \eqref{gtN}.
\end{proposition}

\begin{proof}
For $F\in \mathcal{B}_{b}(  \X \times \mathcal{Z})$, we define the operator $\U^N_{t}:\mathcal{B}_{b}(  \X \times \mathcal{Z})\to \mathcal{B}_{b}(  \X \times \mathcal{Z})$ by
\begin{align}
\U_{t}^{N}F (\u_0,\xi) = F(\Phi_{t}^{N}(\u_0; \xi), R_{t}( \ind_{[t,\infty)}\xi)),  \label{Udef}
\end{align}
and claim that $\{\U_{t}^{N}\}_{t\geq 0}$ is a Markov semigroup. We only need to check the semigroup property. Note that $R_{t}( \ind_{[t,\infty)}\xi)= \ind_{[0,\infty)}R_{t}\xi$ for any $t\in \R$. Given $t_1,t_2\geq 0$, and using \eqref{adapt} and \eqref{semigroup}, we have
\begin{align*}
\U_{t_1+t_2}^{N}F (\u_0 ,\xi) & =F( \Phi^{N}_{t_1+t_2}(\u_0;\xi), R_{t_1+t_2}(\ind_{[t_1+t_2,\infty)}\xi)) \\
& = F( \Phi^{N}_{t_1}( \Phi^{N}_{t_2}(\u_0; \ind_{[0,t_2]}\xi); R_{t_2}(\ind_{[t_2,\infty)}\xi)), R_{t_1} \ind_{[t_1,\infty)}R_{t_2}\xi) \\
&  =F( \Phi^{N}_{t_1}( \Phi^{N}_{t_2}(\u_0;\xi); \ind_{[0,\infty)}R_{t_2}\xi), R_{t_1} \ind_{[t_1,\infty)} \ind_{[0,\infty)}R_{t_2}\xi)\\
& = \U_{t_1}^{N}\big( F (\Phi^N_{t_2}(\u_0;\xi), R_{t_2}\ind_{[t_2,\infty)}\xi)\big) \\
& = (\U^{N}_{t_1}\circ \U_{t_2}^{N}) F (\u_0,\xi).
\end{align*}
This proves the semigroup property. 

We now show \eqref{liftedmeas} and, at the same time, that this density is equal to $g_{t}^{N}$. Note that for any $t >0$, $\ind_{[0,t]}\xi$ and $R_{t}\ind_{[t,\infty)}\xi$ are independent. Then, given $A\subseteq X^{\al}$ and $B\subseteq \mathcal{Z}$ be measurable and using \eqref{gnusN}, we have
\begin{align}
\begin{split}
\int \ind_{A\times B}(\u_0,\xi) d(\U_{t}^{N})^{\ast}(\vec{\nu}_{s,N}\otimes \Xi) & =\int \ind_{A}(\Phi_{t}^{N}(\u_0,\xi)) \ind_{B}(R_{t} \ind_{[t,\infty)}\xi) d\Xi d\vec{\nu}_{s,N} \\
&= \int \E\bigg[ \ind_{A}(\Phi^{N}_{t}(\u_0, \ind_{[0,t]}\xi)) \ind_{B}( R_{t} \ind_{[t,\infty)}\xi)  \bigg] d\vec{\nu}_{s,N} \\
& = \int \E\Big[ \ind_{A}(\Phi^{N}_{t}(\u_0, \ind_{[0,t]}\xi)) \Big] \E\Big[  \ind_{B}( R_{t} \ind_{[t,\infty)}\xi) \Big] d\vec{\nu}_{s,N} \\
& =  \int \ind_{A\times B}(\u_0, \xi)  d((\Pt{t}^{N})^{\ast}\vec{\nu}_{s,N}\otimes \Xi) \\
& = \int \ind_{A\times B}(\u_0, \xi) g_{t}^{N}(\u_0) d(\vec{\nu}_{s,N}\otimes \Xi).
\end{split} \label{FGcomp}
\end{align} 
As this identity is true for all such measurable sets $A$ and $B$, we get that $f_{t}^{N}= g_{t}^{N}$ a.e.
\end{proof}

In order to establish the absolute continuity result, we need to argue locally in the function spaces, which necessitates keeping track of the flow. 
For a fixed measurable set $K\subseteq X^{\al}$ and $t\geq 0$,  we define
\begin{align}
E_{t}^{N,K} : = \big\{ (\u_0, \xi) \,:\, \Phi^{N}_{t'}(\u_0,\xi)\in K \quad \text{for all} \,\, 0\leq t'\leq t\big\}\label{ENK}
\end{align}
These sets behave well with respect to the Markov semigroup $\U_{t}^{N}$ in the following sense.

\begin{lemma}
For every $t,\tau\geq 0$, and $K\subseteq \X$ measurable, it holds that 
\begin{align}
\ind_{E_{t+\tau}^{N,K}}(\u_0,\xi) = \mathcal{U}^{N}_{\tau}( \ind_{E_{t}^{N,K}})(\u_0 ,\xi)\ind_{E_{\tau}^{N,K}}(\u_0,\xi). \label{indineq}
\end{align}
\end{lemma}
\begin{proof}
Suppose that $(\u_0, \xi) \in E_{t+\tau}^{N,K}$.
Using \eqref{Udef}, \eqref{ENK}, and \eqref{semigroup}, we have 
\begin{align*}
\mathcal{U}^{N}_{\tau}( \ind_{E_{t}^{N,K}})(\u_0 ,\xi) = \ind_{E_{t}^{N,K}}( \Phi_{\tau}^{N}(\u_0,\xi), R_{\tau}(\ind_{[\tau,\infty)}\xi))= \ind_{\{ \Phi_{t'}^{N}(\u_0,\xi)\in K \,\, \forall \,\,t'\in [\tau,\tau+t]\}}(\u_0, \xi)=1.
\end{align*}
Moreover, by the inclusion
\begin{align}
E_{t_1}^{N,K}\subseteq E_{t_2}^{N,K} \label{ENKinc}
\end{align}
for any $0\leq t_2\leq t_1$,
we have that $\ind_{E_{\tau}^{N,K}}(\u_0,\xi) =1$. 

Next, assume that $(\u_0,\xi)\notin E_{t+\tau}^{N,K}$. Then, there is a $t_0 \in [0,\tau+t]$ such that $\Phi_{t_0}^{N}(\u_0, \xi)\notin K$. If $t_0 \in [0,\tau]$, then $\ind_{E_{\tau}^{N,K}}(\u_0,\xi)=0$. If instead $t_0\in (\tau,\tau+t]$, then $\mathcal{U}^{N}_{\tau}( \ind_{E_{t}^{N,K}})(\u_0 ,\xi)=0$. 
\end{proof}

We now obtain an analogue of Proposition~\ref{PROP:lift} for the localised measures $ (\vec{\nu}_{s,N}\otimes \Xi) \ind_{E_{t}^{N,K}}$.

\begin{lemma}\label{LEM:ftNK}
Let $N\in \N$, $K\subseteq \X$ be measurable, $t\geq 0$, and $E_t^{N,K}$ be as in \eqref{ENK}. Then, the measure $(\U_{t}^{N})^{\ast}( (\vec{\nu}_{s,N}\otimes \Xi) \ind_{E_{t}^{N,K}})$ has a density with respect to the probability measure $\vec{\nu}_{s,N}\otimes \Xi$ and we set 
\begin{align}
f_{t}^{N,K}(\u_0, \xi)  : = \frac{d (\U_{t}^{N})^{\ast}( (\vec{\nu}_{s,N}\otimes \Xi) \ind_{E_{t}^{N,K}})}{d (\vec{\nu}_{s,N}\otimes \Xi)}. \label{ftNK}
\end{align}
In particular, the following properties hold:

\smallskip
\noi
\textup{(i)} $f_{t}^{N,K}(\u_0,\xi) \geq 0$
for $\vec{\nu}_{s,N}\otimes \Xi$ almost every $(\u_0,\xi)$,

\smallskip
\noi
\textup{(ii)} $f_{t}^{N,K}(\u_0,\xi) = 0$ for almost every $(\u_0,\xi)\in (X^{\al}\setminus K)\times \mathcal{Z}$,

\smallskip
\noi
\textup{(iii)} $f_{t}^{N,K}(\u_0 ,\xi) \leq g_{t}^{N} (\u_0)$
for  almost every $(\u_0,\xi)$.
\end{lemma}

\begin{proof}
The existence of \eqref{ftNK} and consequently (i) is immediate from the existence of $f_{t}^{N}$ which we showed in Proposition~\ref{PROP:lift}. We move onto showing the properties (ii) and (iii).
First, by the definition \eqref{ftNK} we have 
\begin{align}
\int F(\u_0,\xi) f_{t}^{N,K}(\u_0,\xi) d(\vec{\nu}_{s,N}\otimes \Xi) & = \int (\U^{N}_{t}F)(\u_0, \xi) \ind_{E_{t}^{N,K}}(\u_0,\xi) d(\vec{\nu}_{s,N}\otimes \Xi)  \label{propsftNK1}
\end{align}
for any $F\in \mathcal{B}_{b}(X^{\al}\times \mathcal{Z})$. 
By choosing $F=\ind_{X^{\al}\setminus K}(\u_0)$ in \eqref{propsftNK1} and noting that $E_{t}^{N,K}$ is disjoint from $\{\Phi^{N}_{t}(\u_0,\xi)\in (X^{\al}\setminus K)\times \mathcal{Z}\}$, we obtain
\begin{align}
\int \ind_{X^{\al}\setminus K}(\u_0)    f_{t}^{N,K}(\u_0,\xi) d(\vec{\nu}_{s,N}\otimes \Xi)=0.\label{propsftNK2}
\end{align} 
Property (ii) now follows from property (i) and that $f_{t}^{N,K}\geq 0$.

We move onto showing property (iii). For $F\in \mathcal{B}_{b}(X^{\al}\times \mathcal{Z})$ such that $F\geq 0$, \eqref{propsftNK1}, \eqref{ENKinc} and \eqref{Udef}, we have
\begin{align}
\begin{split}
\int F(\u_0,\xi) f_{t}^{N,K}(\u_0,\xi) d(\vec{\nu}_{s,N}\otimes \Xi) & = \int (\U^{N}_{t}F)(\u_0, \xi) \ind_{E_{t}^{N,K}}(\u_0,\xi) d(\vec{\nu}_{s,N}\otimes \Xi) \\
& \leq  \int (\U^{N}_{t}F)(\u_0, \xi) \ind_{K}(\Phi_{t}^{N}(\u_0,\xi)) d(\vec{\nu}_{s,N}\otimes \Xi) \\
& = \int \U^{N}_{t}(F \ind_{K})(\u_0, \xi) d(\vec{\nu}_{s,N}\otimes \Xi) \\
&=  \int F(\u_0,\xi) \ind_{K}(\u_0) f_{t}^{N}(\u_0,\xi)d(\vec{\nu}_{s,N}\otimes \Xi) \\
& =  \int F(\u_0,\xi) \ind_{K}(\u_0) g_{t}^{N}(\u_0)d(\vec{\nu}_{s,N}\otimes \Xi),
\end{split} \label{Fineq}
\end{align}
which proves (iii).
\end{proof}

\subsection{$L^p$ bounds on densities}

The key ingredient for the proof of Theorem~\ref{THM:QI} is the following long-time $L^p$ bounds on the truncated densities $f_{t}^{N,K}$.

\begin{proposition}[Uniform long-time bounds on densities]
Let $s>0$, $0<\al<s$ be as in Lemma~\ref{LEM:gtN}, $T>0$, $R>0$, and $K_{R}$ be as in \eqref{KR}. 
Then, it holds that
\begin{align}
\sup_{N\in 2^{\N_0}} \sup_{t\in [0,T]} \|f_{t}^{N,K_{R}} \|_{L^{q} (\vec{\nu}_{s,N}\otimes \Xi)} \leq C(R,T,q) \label{ftNKLp}
\end{align}
for any $1\leq q<\infty$.
\end{proposition}

\begin{proof} 
To simplify the notation, we write $K$ in place of $K_{R}$.
First, we will prove that 
\begin{align}
\begin{split}
\| f_{t+\tau}^{N,K}\|_{L^{r}(\vec{\nu}_{s,N}\otimes \Xi)}   \leq \big\| f_{\tau}^{N,K} \big\|_{L^p (\vec{\nu}_{s,N}\otimes \Xi)} \big\| f_{t}^{N,K}\big\|_{L^{\frac{r(p-1)}{p-r}}(\vec{\nu}_{s,N}\otimes \Xi)}^{\frac{1}{p'}}, 
\end{split}\label{densitydiff}
\end{align}
for any $0<\tau,t<T$ such that $t+\tau\leq T$ and $1\leq r<p<\infty$. 

Let $F\in C_{b}(\X\times \mathcal{Z})$. Then, by \eqref{indineq}, the semigroup property of $\U_{t}^{N}$, \eqref{ENKinc}, and H\"{o}lder's inequality,
\begin{align}
\int & F(\u_0 , \xi)  f_{t+\tau}^{N,K}(\u_0,\xi)d( \vec{\nu}_{s,N}\otimes \Xi) \notag \\
& = \int (\U^N_{t+\tau}F )(\u_0, \xi) \ind_{E^{N,K}_{t+\tau}}(\u_0, \xi) d ( \vec{\nu}_{s,N}\otimes \Xi) \notag   \\
& = \int \U^{N}_{\tau}( \U^{N}_{t}F)(\u_0,\xi)\U^{N}_{\tau}(\ind_{E^{N,K}_{t}})(\u_0,\xi) \ind_{E^{N,K}_{\tau}}(\u_0,\xi) d( \vec{\nu}_{s,N}\otimes \Xi) \notag \\
& = \int \U^{N}_{\tau}(( \U^{N}_{t}F)\ind_{E^{N,K}_{t}})(\u_0,\xi) \ind_{E^{N,K}_{\tau}}(\u_0,\xi) d( \vec{\nu}_{s,N}\otimes \Xi)  \notag \\
& = \int f_{\tau}^{N,K}(\u_0,\xi)(\U^{N}_{t}F)(\u_0,\xi)\ind_{E^{N,K}_{t}}(\u_0,\xi)d ( \vec{\nu}_{s,N}\otimes \Xi)  \notag \\
& \leq \big\| f_{\tau}^{N,K}\big\|_{L^p ( \vec{\nu}_{s,N}\otimes \Xi)  } \big\| (\U^{N}_{t}F)\ind_{E^{N,K}_{t}}\big\|_{L^{p'} ( \vec{\nu}_{s,N}\otimes \Xi)  }. \label{gronwall1}
\end{align}
Now using the definition \eqref{Udef} and \eqref{indineq}, we have
\begin{align}
\big\|  (\U^{N}_{t}F )\ind_{E^{N,K}_{t}}\big\|_{L^{q}(\vec{\nu}_{s,N}\otimes \Xi)}^{q} 
& = \int \U_{t}^{N}(|F|^{q})\ind_{E_{t}^{N,K}}(\u_0,\xi) d(\vec{\nu}_{s,N}\otimes \Xi) \notag\\
& = \int |F(\u_0,\xi)|^{q}  f_{t}^{N,K}(\u_0,\xi)d(\vec{\nu}_{s,N}\otimes \Xi). \label{gronwall2}
\end{align}
 for any $1\leq q<\infty$.
Using \eqref{gronwall2} in \eqref{gronwall1}, we then have
\begin{align*}
\int & F(\u_0 , \xi)  f_{t+\tau}^{N,K}(\u_0,\xi)d( \vec{\nu}_{s,N}\otimes \Xi)\\
& \leq \big\| f_{\tau}^{N,K}\big\|_{L^p ( \vec{\nu}_{s,N}\otimes \Xi)  }   \big\| F\cdot (f_{t}^{N,K})^{\frac{1}{p'}}\big\|_{L^{p'}( \vec{\nu}_{s,N}\otimes \Xi)}\\
& \leq \big\| f_{\tau}^{N,K}\big\|_{L^p ( \vec{\nu}_{s,N}\otimes \Xi)  }    
\|F\|_{L^{r'}(\vec{\nu}_{s,N}\otimes \Xi)} 
\big\| (f_{t}^{N,K})^{\frac{1}{p'}}\big\|_{L^{\frac{rp}{p-r}}  (\vec{\nu}_{s,N}\otimes \Xi)} \\
&=  \big\| f_{\tau}^{N,K}\big\|_{L^p ( \vec{\nu}_{s,N}\otimes \Xi)  }  
\|F\|_{L^{r'}(\vec{\nu}_{s,N}\otimes \Xi)} 
 \big\| f_{t}^{N,K} \big\|_{L^{\frac{r(p-1)}{p-r}}(\vec{\nu}_{s,N}\otimes \Xi)}^{\frac{1}{p'}}.
\end{align*}
Thus, by duality we establish \eqref{densitydiff}.

We now iterate \eqref{densitydiff}. We choose the time-step $0<\tau=\tau_{R}:=\tau(K_{R},N)\leq 1$ from Lemma~\ref{LEM:gtN} and break the time interval $[0,T]$ into $J=\left\lceil \frac{T}{\tau} \right\rceil $ many sub-intervals of width $\tau_{R}>0$. We let $t_{j}:= j\tau_{R}$ for $j=0,1,\ldots, J$, and note that $0\leq j\tau_{R} \leq T$. 
For $\eps_0>0$ to be chosen sufficiently small, we define $p:= \frac{1}{\eps_0 \tau_{R}}>1$ and the sequence $\{r_{j}\}_{j=1}^{J}$ recursively by
\begin{align*}
r_{1}=3q \quad \text{and} \quad   r_{j+1} = \frac{r_{j}\, p}{p+r_{j}-1}.
\end{align*}
Note that $1<r_{j+1}<r_{j}$ for all $j=1,\ldots, J-1$ and $r_{j}=\frac{r_{j+1}(p-1)}{p-r_{j+1}}$.
Using that the sequence $\{r_{j}\}_{j=1}^{J}$ is decreasing, we have 
\begin{align*}
r_{1}-r_{J} = \sum_{j=1}^{J-1} (r_{j}-r_{j+1}) =\sum_{j=1}^{J-1} \frac{r_j (r_{j}-1)}{p+r_{j}-1} \leq \frac{r_{1}^{2}}{p}J  \leq r_{1}^{2} \eps_0 T.
\end{align*}
Thus, by choosing 
\begin{align*}
\eps_0=\eps_0(q,T) =\min\big( \tfrac{1}{2q-1}, \tfrac{\min(9T,1)}{9T q}\big)>0,
\end{align*}
we ensure that $p\geq 2q-1\geq q$ and $r_{1}-r_{J}\leq q$. In particular, $r_{J}\geq 2q$.
For $j=1,\ldots,J-1$, \eqref{densitydiff} implies
\begin{align*}
\| f^{N,K}_{t_{j+1}}\|_{L^{r_{j+1}}(\vec{\nu}_{s,N}\otimes \Xi)} 
&\leq  \|f_{\tau_{R}}^{N,K} \|_{L^p (\vec{\nu}_{s,N}\otimes \Xi)} \big\| f_{t_j}^{N,K}\big\|_{L^{r_{j}}(\vec{\nu}_{s,N}\otimes \Xi)}^{\frac{1}{p'}}
\end{align*}
Therefore,
\begin{align*}
\| f^{N,K}_{t_{j}}\|_{L^{r_{j}}(\vec{\nu}_{s,N}\otimes \Xi)} \leq \|f_{\tau_{R}}^{N,K} \|_{L^p (\vec{\nu}_{s,N}\otimes \Xi)}^{\sum_{k=0}^{j} \frac{1}{(p')^{k}}},
\end{align*}
for each $j=1,\ldots, J$. As
\begin{align*}
\sum_{k=0}^{\infty} \frac{1}{(p')^{k}} = \frac{1}{1-\frac{1}{p}}=p,
\end{align*}
 H\"{o}lder's inequality then implies
\begin{align}
\max_{j=1,\ldots, J} \| f^{N,K}_{t_{j}}\|_{L^{2q}(\vec{\nu}_{s,N}\otimes \Xi)}  \leq \max\big(1, \|f_{\tau_{R}}^{N,K} \|_{L^p (\vec{\nu}_{s,N}\otimes \Xi)}^{J}\big). \label{ftNKbd1}
\end{align} 
Now recalling that we chose $p$ such that $p\tau_{R} =\eps_0^{-1}$ and using Lemma~\ref{LEM:ftNK} (ii)-(iii), Proposition~\ref{PROP:lift}, and \eqref{gbd}, we have
\begin{align}
\begin{split}
\|f_{\tau_{R}}^{N,K_{R}} \|_{L^p (\vec{\nu}_{s,N}\otimes \Xi)}^{J} 
&\leq  \| g_{\tau_{R}}^{N} \ind_{K_{R}}(\u_0) \|_{L^{p}(\vec{\nu}_{s,N})}^{J}  \\
& \leq  \| g_{\tau_{R}}^{N} \ind_{\Pi_{\leq N}K_{R}}(\Pi_{\leq N}\u_0) \|_{L^{p}(\vec{\nu}_{s,N})}^{J}\\
&= \big(  \| g_{\tau_{R}}^{N} \ind_{K_{R,N}}( \Pi_{\leq N}\u_0) \|_{L^{p}(\vec{\nu}_{s,N})}^{p} \big)^{\frac{J}{p}} \\
& \leq  c_{0}^{\eps_0 T} \exp( C(R) T\eps_0^{-1}) =: \wt{C}(R,T,q),
\end{split} \label{fontauR}
\end{align}
uniformly in $N\in 2^{\N_0}$.  Returning this bound to \eqref{ftNKbd1}, establishes that
\begin{align}
\max_{j=1,\ldots, J} \| f^{N,K}_{t_{j}}\|_{L^{2q}(\vec{\nu}_{s,N}\otimes \Xi)} \leq \wt{C}(R,T,q) \label{discretebd}
\end{align}
uniformly in $N\in 2^{\N_0}$.  

Now we obtain $L^q$ bounds for \textit{every} $t\in [0,T]$. 
Let $t\in [0,T]$ and assume that $t$ is not an integer multiple of $\tau_{R}$. If $t\in (0,\tau_{R})$, then  Lemma~\ref{LEM:ftNK} (ii)-(iii), Proposition~\ref{PROP:lift}, \eqref{gbd} imply
\begin{align}
\sup_{t\in [0,\tau_{R}]} \|f_{t}^{N,K}\|_{L^p (\vec{\nu}_{s,N}\otimes \Xi)} \leq \wt{C}(R,T,q) \label{ctsbd}
\end{align}
uniformly in $N\in 2^{\N_0}$.
Now we assume that $t>\tau_{R}$. There is a $k\in \N$ (depending on $N$) such that $k\tau_{R} <t<(k+1)\tau_{R}$. By repeating the arguments leading to \eqref{densitydiff}, we find that 
\begin{align*}
\| f_{t}^{N,K}\|_{L^{r}(\vec{\nu}_{s,N}\otimes \Xi)} 
\leq \| f_{t-k\tau_{R}}^{N,K}\|_{L^p (\vec{\nu}_{s,N}\otimes \Xi)} \| f_{k\tau_{R}}^{N,K}\|_{L^{{\frac{r(p-1)}{p-r}}}(\vec{\nu}_{s,N}\otimes \Xi)}^{\frac{1}{p'}}
\end{align*}
for any $1\leq r<p<\infty$. We choose $r=\frac{2qp}{p+2q-1}$ so that $\frac{r(p-1)}{p-r}=2q$. Moreover, since $p\geq 2q-1$, we have $r\geq q$. It now follows from \eqref{discretebd}, \eqref{ctsbd}, and H\"{o}lder's inequality that
\begin{align}
\sup_{N\in 2^{\N_0}}\| f_{t}^{N,K}\|_{L^{q}(\vec{\nu}_{s,N}\otimes \Xi)}  \leq \wt{C}(R,T,q)^{2} \label{ctsbd2}
\end{align}
 As $t\in [0,T]$ was arbitrary, \eqref{ctsbd} and \eqref{ctsbd2} complete the proof of \eqref{ftNKLp}.
\end{proof}

\subsection{Proof of Theorem~\ref{THM:QI}}

Let $s>0$, $0<\al<s$ be as in \textup{Lemma}~\ref{LEM:gtN}. Let $T>0$. For $L\geq 1$, define the set 
\begin{align*}
D_{T,L} &: = \Big\{ (\u_0, \xi)\in  X^{\al}\times \mathcal{Z} \,:\, \\
& \sup_{M\in 2^{\N_0}}\sup_{t\in [0,T]} \Big( \| Q_{s,M}(\pi_{1} \Phi^{\text{lin}}_{t}(\u_0,\xi))\|_{H^{\al-s}} + \|\Pi_{\leq M}\Phi_{t}^{\text{lin}}(\u_0,\xi)\|_{\CC^{\al}} + \| \Pi_{\leq M}\stick_{t}(\xi)\|_{Z^{\al}}\Big) \leq L\Big\}.
\end{align*}
and let $0<\al'<\al$ be sufficiently close to $\al$ so that $(s,\al')$ also satisfy the conditions in \textup{Lemmas}~\ref{LEM:Calcontrol}, \ref{LEM:Qscontrol}, and \ref{LEM:gtN}.
Then, for any $F\in C_{b}(\H^{\al'}(\T^2))$ such that $F\geq 0$, Lemma~\ref{LEM:ASconvflow}, and Lemma~\ref{LEM:nus} imply
\begin{align}
\begin{split}
\int F(\u_0) d (\Pt{t})^{\ast} \vec{\nu}_s
& = \int (\Pt{t}F)(\u_0)  d\vec{\nu}_{s} \\
& =\int \E\big[ F(\Phi_{t}(\u_0,\xi)) \big]  d\vec{\nu}_{s} \\
&  =\lim_{N\to \infty} \int \E\big[ F(\Phi^{N}_{t}(\u_0,\xi)) \big]  d\vec{\nu}_{s}  \\
& =\lim_{N\to\infty}  \int \E\big[ F(\Phi^{N}_{t}(\u_0,\xi)) \big] d\vec{\nu}_{s,N}  \\
& = \lim_{N\to\infty} \int F(\Phi_{t}^{N}(\u_0,\xi))  d(\vec{\nu}_{s,N}\otimes \Xi) \\
& \leq  \limsup_{N\to\infty} \int \ind_{D_{T,L}}(\u_0, \xi) F(\Phi_{t}^{N}(\u_0,\xi)) d(\vec{\nu}_{s,N}\otimes \Xi) \\
& \hphantom{X} +\limsup_{N\to\infty} \int \ind_{D_{T,L}^{c}}(\u_0,\xi) F(\Phi_{t}^{N}(\u_0,\xi)) d(\vec{\nu}_{s,N}\otimes \Xi).
\end{split} \label{Ptbds0}
\end{align}
It follows by Markov's inequality with \eqref{uN}, \eqref{momentsstick}, \eqref{uniflin}, \eqref{Qslin2} that 
\begin{align}
\bigg|\int \ind_{D_{T,L}^{c}}(\u_0,\xi) F(\Phi_{t}^{N}(\u_0,\xi))  d(\vec{\nu}_{s,N}\otimes \Xi)  \bigg| \leq  \frac{C(T,s,\al) \|F\|_{L^{\infty}}}{L}. \label{DTLc}
\end{align}
From Lemma~\ref{LEM:Calcontrol} and Lemma~\ref{LEM:Qscontrol}, we see that if $(\u_0,\xi)\in D_{T,L}$, then $(\u_0,\xi) \in E^{N,\wt{K}_{C(L,T)}}_{t}$, where $\wt{K}_{C(L,T)}$ is a set of the same form as $K_{R}$ in \eqref{KR} but with radius $C(L,T)$ and $\al$ replaced by $\al'$. Consequently, 
\begin{align}
\begin{split}
\int & \ind_{D_{T,L}}(\u_0, \xi) F(\Phi_{t}^{N}(\u_0,\xi)) d(\vec{\nu}_{s,N}\otimes \Xi) \\
& \leq \int \ind_{E^{N,\wt{K}_{C(L,T)}}_{t}}(\u_0,\xi) F(\Phi_{t}^{N}(\u_0,\xi))  d(\vec{\nu}_{s,N}\otimes \Xi)\\
& = \int F(\u_0) f_{t}^{N,\wt{K}_{C(L,T)}}(\u_0,\xi) d(\vec{\nu}_{s,N}\otimes \Xi).
\end{split} \label{Ptbds1}
\end{align}
It follows from \eqref{ftNKLp} that, up to a subsequence, 
\begin{align*}
f_{t}^{N,\wt{K}_{C(R,L,T)}}(\u_0,\xi) Z_{s,N}^{-1}e^{-R_{s,N}(\u_0)}
\end{align*}
 converges weakly in $L^2(d(\vec{\mu}_{s}\otimes \Xi))$ to some element $\mathfrak{f}_{t}^{\wt{K}_{C(L,T)}}\in L^2(d(\vec{\mu}_{s}\otimes \Xi))$. 
It then follows from \eqref{Ptbds0}, \eqref{DTLc}, and \eqref{Ptbds1} that
\begin{align}
\int F(\u_0) d (\Pt{t})_{\ast}\vec{\nu}_s \leq \int F(\u_0) \mathfrak{f}_{t}^{\wt{K}_{C(L,T)}}(\u_0,\xi) d(\vec{\mu}_{s}\otimes \Xi) + \tfrac{C(T,s,\al)\|F\|_{L^{\infty}}}{L}. \label{Ptbds2}
\end{align}
By density of continuous, bounded functions in
 $$L^1\big( (\Pt{t})_{\ast}\vec{\nu}_s+(\pi_{1})_{\ast} \mathfrak{f}_{t}^{\wt{K}_{C(L,T)}}(\u_0,\xi) d(\vec{\mu}_{s}\otimes \Xi)\big),$$
we can extend \eqref{Ptbds2} to hold for all $F\in \mathcal{B}_{b}(X^{\al})$. Let $E\subset X^{\al}$ be a Borel set such that $\vec{\mu}_{s}(E)=0$. Therefore, by \eqref{Ptbds2}, we obtain
\begin{align*}
\int \ind_{E}(\u_0) d (\Pt{t})_{\ast}\vec{\nu}_s  \leq \tfrac{C(T,s,\al)}{L},
\end{align*}
for every $L\ge 1$. By taking $L\to \infty$, we obtain that $(\Pt{t})_{\ast}\vec{\nu}_{s}(E)=0$. Since this holds for every $E$ with $\vec{\mu}_s(E) = 0$, we deduce that $(\Pt{t})_{\ast}\vec{\mu}_{s} \ll (\Pt{t})_{\ast}\vec{\nu}_{s} \ll \vec{\mu}_s$.
This completes the proof of Theorem~\ref{THM:QI}.

\begin{ackno}\rm
J.\,F. acknowledges support from Tadahiro Oh's ERC consolidator grant
	(grant no. 864138 ``SingStochDispDyn'').
\end{ackno}


\begin{thebibliography}{99}





\bibitem{BG}
N.~Barashkov, M.~Gubinelli,
{\it  A variational method for $\Phi^4_3$}, 
Duke Math. J. 169 (2020), no. 17, 3339--3415.

\bibitem{BG2}
N. Barashkov, M. Gubinelli,
{\it The $\Phi^{4}_{3}$ measure via Girsanov's theorem},
Electron. J. Probab. 26(2021), Paper No. 81, 29 pp.

\bibitem{BCD}
H. Bahouri, J.Y. Chemin, R. Danchin, {\it Fourier analysis and nonlinear partial differential equations}, Vol. 343. Springer Science \& Business Media, 2011.

\bibitem{BKRS}
V. I. Bogachev, N. V. Krylov, M. Röckner, S. V. Shaposhnikov, 
{\it Fokker-Planck-Kolmogorov equations},
Math. Surveys Monogr., 207
American Mathematical Society, Providence, RI, 2015. xii+479 pp.


  \bibitem{BD}
M.~Bou\'e, P.~Dupuis,
{\it A variational representation for certain functionals of Brownian motion},
Ann. Probab. 26 (1998), no. 4, 1641--1659.

  \bibitem{Bo96}
J. Bourgain, {\it Invariant measures for the 
2D-defocusing nonlinear Schr\"odinger equation},
Comm. Math. Phys. 176(2): 421-445 (1996).




\bibitem{BTh}
N.~Burq, L.~Thomann,
{\it Almost sure scattering for the one dimensional nonlinear Schr\"{o}dinger equation},
Mem. Amer. Math. Soc. 296(2024), no.1480, vii+87 pp.

\bibitem{BuTz2}
N. Burq, N. Tzvetkov, {\it Random data Cauchy theory for supercritical wave equations. II. A global existence result}, Invent. Math. 173 (2008), no. 3, 477--496.


\bibitem{cmw}
A.~Chandra, A.~Moinat, H.~Weber, \emph{A priori bounds for the $\Phi^4$ equation in the full sub-critical regime}, 
Arch. Ration. Mech. Anal.247(2023), no.3, Paper No. 48, 76 pp.



\bibitem{CT}
J. Coe, L. Tolomeo,
{\it Sharp quasi-invariance threshold for the cubic Szeg\"{o} equation}, 
	arXiv:2404.14950 [math.AP].
 
%
%

\bibitem{dpd}
 G.~Da Prato, A.~Debussche, 
 {\it Two-dimensional Navier-Stokes equations driven by a space-time white
noise}, J. Funct. Anal. 196 (2002), no. 1, 180--210.
%


\bibitem{DPZ1}
G.~Da Prato, J.~Zabczyk, 
{\it Stochastic equations in infinite dimensions},
Second edition. Encyclopedia of Mathematics and its Applications, 152. Cambridge University Press, Cambridge, 2014. xviii+493 pp. ISBN: 978-1-107-05584-1.




\bibitem{DT2020} 
A.~Debussche, Y.~Tsutsumi,
{\it  Quasi-invariance of Gaussian measures transported by the
	cubic NLS with third-order dispersion on $\mathbf T$}, 
J. Funct. Anal. 281 (2021), no. 3, 109032, 23 pp.

\bibitem{EGZ}
A. Eberle, A. Guillin, R. Zimmer, \emph{Couplings and quantitative contraction rates for Langevin dynamics}, Ann. Probab. 47 (2019), no. 4, 1982--2010.

\bibitem{Fer97} 
B. Ferrario, \emph{Ergodic results for stochastic Navier-Stokes equation}, Stochastics and Stochastics Reports 60 (1997), 271--288.

\bibitem{FM95}
F.~Flandoli, B.~Maslowski,
{\it Ergodicity of the 2-D Navier-Stokes equation under random perturbations},
Comm. Math. Phys. 172 (1995), no. 1, 119--141.

	
\bibitem{FS}
J.~Forlano, K.~Seong,
{\it Transport of Gaussian measures under the flow of one-dimensional fractional nonlinear Schr\"{o}dinger equations,} Comm. PDE, 47 (2022), no. 6, 1296--1337.

\bibitem{FT1}
J. Forlano, L. Tolomeo,
{\it On the unique ergodicity for a class of 2 dimensional stochastic wave equations}, Trans. Amer. Math. Soc., 377 (2024), 345--394.

\bibitem{FT2}
J. Forlano, L. Tolomeo,
{\it Quasi-invariance of Gaussian measures of negative regularity for fractional nonlinear Schr\"{o}dinger equations}, arXiv:2205.11453 [math.AP].


\bibitem{FT}
 J.~Forlano, W.~J.~Trenberth, 
 {\it On the transport of Gaussian measures under the one-dimensional fractional nonlinear Schr\"odinger equations,} 
	Ann. Inst. H. Poincar\'e Anal. Non Lin\'eaire 36 (2019), no. 7, 1987--2025.


\bibitem{GLT1}
G.~Genovese, R.~Luc\'{a}, N.~Tzvetkov,
	{\it Transport of Gaussian measures with exponential cut-off for Hamiltonian PDEs},
J. Anal. Math.150(2023), no.2, 737--787.

\bibitem{GLT2}
G.~Genovese, R.~Luc\'{a}, N.~Tzvetkov,
{\it Quasi-invariance of Gaussian measures for the periodic Benjamin-Ono-BBM equation}, 
Stoch. Partial Differ. Equ. Anal. Comput.11(2023), no.2, 651--684.


\bibitem{gh1}
M.~Gubinelli, M.~Hofmanov'a, \emph{Global solutions to elliptic and parabolic $\Phi^4$ models in Euclidean space}, Comm. Math. Phys. 368 (2019), no. 3, 1201--1266.


\bibitem{gh2}
M.~Gubinelli, M.~Hofmanov'a, \emph{A PDE construction of the Euclidean $\Phi^4_3$ quantum field theory}, Comm. Math. Phys. 384 (2021), 1--75.



\bibitem{GKO}
 M.~Gubinelli, H.~Koch, T.~Oh, {\it Renormalization of the two-dimensional stochastic nonlinear wave equations,}
 Trans. Amer. Math. Soc. 370 (2018), no 10, 7335--7359.


\bibitem{GKO2}
M. Gubinelli,  H. Koch, T. Oh, 
{\it  Paracontrolled approach to the three-dimensional stochastic nonlinear wave equation with quadratic nonlinearity},
J. Eur. Math. Soc. 26 (2024), no. 3, 817--874.

\bibitem{GKOT}
 M.~Gubinelli, H.~Koch, T.~Oh, L.~Tolomeo,
  {\it Global dynamics for the two-dimensional stochastic nonlinear wave equations,}
Int. Math. Res. Not. (2021), rnab084, https://doi.org/10.1093/imrn/rnab084.



\bibitem{GOTW}
T.~Gunaratnam, T.~Oh, N.~Tzvetkov, H.~Weber,
{\it  Quasi-invariant Gaussian measures for the nonlinear wave equation in three dimensions},
Probab. Math. Phys. 3 (2022), no. 2, 343--379.

 



\bibitem{HKN}
M. Hairer, S. Kusuoka, H. Nagoji,
{\it Singularity of solutions to singular SPDEs},
arXiv:2409.10037 [math.PR]



%

%

\bibitem{KMV}
R. Killip, J. Murphy, M. Vi\c{s}an, {\it Invariance of white noise for KdV on the line}, Invent.
Math. 222 (2020), no. 1, 203--282.

\bibitem{Knezevitch}
A. Knezevitch,
{\it Transport of low regularity Gaussian measures for the 1d quintic nonlinear Schr\"{o}dinger equation}
	arXiv:2406.07116 [math.AP].


%
%

\bibitem{MS}
J. C. Mattingly, T. M. Suidan,
{\it The small scales of the stochastic Navier-Stokes equations under rough forcing},
J. Stat. Phys. 118 (2005), no. 1--2, 343--364.

\bibitem{mw2}
J.-C.~Mourrat, H.~Weber, \emph{The dynamic $\Phi^4_3$ model comes down from infinity}, Comm. Math. Phys. 356 (2017), no. 3, 673--753.


 \bibitem{mckean}
 H.\,P.~McKean, 
 {\it Statistical mechanics of nonlinear wave equations. IV. Cubic Schr\"odinger}, Comm. Math.
Phys. 168 (1995), no. 3, 479--491. {\it Erratum: Statistical mechanics of nonlinear wave equations. IV. Cubic
Schr\"odinger,} Comm. Math. Phys. 173 (1995), no. 3, 675.



\bibitem{OOT1}
T. Oh, M. Okamoto, L. Tolomeo,
{\it Focusing $\Phi^{4}{3}$ model with a Hartree nonlinearity},
to appear in Mem. Amer. Math. Soc.

\bibitem{OOT2}
T. Oh, M. Okamoto, L. Tolomeo,
{\it Stochastic quantisation of the $\Phi^{3}{3}$-model},
to appear in Mem. Eur. Math. Soc.


\bibitem{OS}
T.~Oh,  K.~Seong, 
{\it  Quasi-invariant Gaussian measures for the cubic fourth order 
	nonlinear Schr\"odinger equation 
	in negative Sobolev spaces,}
J. Funct. Anal. 281 (2021), no. 9, 109150, 49 pp.

\bibitem{OST}
T.~Oh, P.~Sosoe, N.~Tzvetkov,
{\it  An optimal regularity result on the quasi-invariant Gaussian measures for the cubic fourth order nonlinear 
	Schr\"odinger equation,} 
J. \'Ec. polytech. Math. 5
(2018), 793--841. 


\bibitem{OTT}
T.~Oh, Y.~Tsutsumi, N.~Tzvetkov,
{\it  Quasi-invariant Gaussian measures for the cubic nonlinear Schr\"odinger equation with third order dispersion}, C. R. Math. Acad. Sci. Paris 357 (2019), no. 4, 366--381. 

  \bibitem{OTz}
T.~Oh, N.~Tzvetkov,
{\it Quasi-invariant Gaussian measures for the cubic fourth order nonlinear Schr\"odinger equation}, 
Probab. Theory Related Fields 169 (2017), 1121--1168. 


\bibitem{OTz2}
T.~Oh, N.~Tzvetkov,
{\it Quasi-invariant Gaussian measures for the two-dimensional defocusing cubic nonlinear wave
	equation}, 
J. Eur. Math. Soc. 22 (2020), no. 6, 1785--1826.

\bibitem{OWZ}
T. Oh, Y. Wang, Y. Zine,
{\it Three-dimensional stochastic cubic nonlinear wave equation with almost space-time white noise},
Stoch. Partial Differ. Equ. Anal. Comput.10(2022), no.3, 898--963.


\bibitem{PTV} 
F.~Planchon, N.~Tzvetkov, N.~Visciglia, 
{\it  Transport of Gaussian measures by the 
	flow of the nonlinear Schr\"odinger equation}, 
Math. Ann. 378 (2020), no. 1-2, 389--423.



\bibitem{PTV2} 
F.~Planchon, N.~Tzvetkov, N.~Visciglia, 
{\it Modified energies for the periodic generalized KdV equation and applications},
Ann. Inst. H. Poincaré C Anal. Non Linéaire 40 (2023), no.4, 863--917.


\bibitem{STX}
P.~Sosoe, W.~Trenberth, T.~Xian,
{\it Quasi-invariance of fractional Gaussian fields by the nonlinear wave equation with polynomial nonlinearity}, 
Differential Integral Equations 33 (2020), no. 7-8, 393--430.

\bibitem{Simon}
B. Simon, {\it The $P(\Phi)2$ Euclidean (quantum) field theory}, Princeton Series in Physics. Princeton University
Press, Princeton, N.J., 1974. xx+392 pp




\bibitem{RevuzYor}
D.~Revuz, M.~Yor, 
{\it Continuous martingales and Brownian motion.
Third edition}, Grundlehren der Mathematischen Wissenschaften [Fundamental Principles of Mathematical Sciences], 293. Springer-Verlag, Berlin, 1999. xiv+602 pp. ISBN: 3-540-64325-7.
	
\bibitem{ST}
C.~Sun, N.~Tzvetkov,
{\it Quasi-invariance of Gaussian measures for the 3d energy critical nonlinear Schr\"{o}dinger equation},
	arXiv:2308.12758 [math.AP].
	
	
	\bibitem{ThTz}
L. Thomann, N. Tzvetkov, {\it Gibbs measure for the periodic derivative nonlinear Schr\"{o}dinger equation},
Nonlinearity 23 (2010), no. 11, 2771--2791.
	

\bibitem{t18wave}
L.~Tolomeo, 
{\it Global well-posedness of the two-dimensional stochastic nonlinear wave equation on an unbounded domain}, 
Ann. Probab. 49(3): 1402-1426 (May 2021).

 
\bibitem{t18erg}
L.~Tolomeo, 
{\it 
Unique ergodicity for a class of stochastic hyperbolic equations with additive space-time white noise}, 
 Comm. Math. Phys. 377 (2020), no. 2, 1311--1347.


\bibitem{Tolomeo3}
L.~Tolomeo, 
{\it Ergodicity for the hyperbolic $P(\Phi)_2$-model}, 
arXiv:2310.02190 [math.PR].



\bibitem{Tz}
N.~Tzvetkov, {\it Quasi-invariant Gaussian measures for one dimensional Hamiltonian PDE's,}
Forum Math. Sigma 3 (2015), e28, 35 pp. 


\bibitem{Ust}
A.~ \"Ust\"unel,
{\it Variational calculation of Laplace transforms via entropy on Wiener space and applications},
J. Funct. Anal. 267 (2014), no. 8, 3058--3083.

\end{thebibliography}
\end{document}